\documentclass[a4paper,reqno, 10pt]{amsart}

\usepackage{amssymb}
\usepackage{amstext}
\usepackage{amsmath}
\usepackage{amscd}
\usepackage{amsthm}
\usepackage{amsfonts}
\usepackage{dsfont}
\usepackage{enumerate}
\usepackage{graphicx}
\usepackage{latexsym}
\usepackage{mathrsfs}
\usepackage[all]{xy}
\xyoption{all}

\usepackage{pstricks}
\usepackage{lscape}

\newtheorem{theorem}{Theorem}[section]

\newtheorem{corollary}[theorem]{Corollary}
\newtheorem{lemma}[theorem]{Lemma}
\newtheorem{proposition}[theorem]{Proposition}
\newtheorem{definition-proposition}[theorem]{Definition-Proposition}

\theoremstyle{definition}
\newtheorem{definition}[theorem]{Definition}

\theoremstyle{remark}
\newtheorem{remark}[theorem]{Remark}

\newcommand{\CH}{\operatorname{CH}\nolimits}

\renewcommand{\hom}{\mathrm{hom}}
\newcommand{\alg}{\mathrm{alg}}
\newcommand{\C}{\mathbf{C}}

\newcommand{\Q}{\mathbf{Q}}

\newcommand{\X}{\widetilde{X}}

\newcommand{\im}{\mathrm{Im}\,}

\newcommand{\End}{\mathrm{End}}

\newcommand{\id}{\mathrm{id}}

\newcommand{\Spec}{\mathrm{Spec} \,}

\renewcommand\P{\mathbf{P}}

\renewcommand\r{\rightarrow}

\renewcommand{\ker}{\mathrm{Ker}\,}

\newcommand\h{\mathrm{H}}

\renewcommand\h{\mathfrak{h}}

\renewcommand\j{\widetilde{j}}

\begin{document}

\title{Algebraic cycles and fibrations}
\author{Charles Vial}

\address{DPMMS, University of Cambridge, Wilberforce Road, Cambridge,
  CB3 0WB, UK} \email{c.vial@dpmms.cam.ac.uk}
\urladdr{http://www.dpmms.cam.ac.uk/~cv248/} 
\thanks{ The author is supported by  EPSRC Early Career Fellowship
  number EP/K005545/1.}

\maketitle



For a scheme $X$ over a field $k$, $\CH_i(X)$ denotes the rational
Chow group of $i$-dimensional cycles on $X$ modulo rational
equivalence.  Throughout, $f : X \r B$ will be a projective surjective
morphism defined over $k$ from a quasi-projective variety $X$ of
dimension $d_X$ to an irreducible quasi-projective variety $B$ of
dimension $d_B$, with various extra assumptions which will be
explicitly stated.  Let $h$ be the class of a hyperplane section in
the Picard group of $X$.  Intersecting with $h$ induces an action
$\CH_i(X) \r \CH_{i-1}(X)$ still denoted $h$. Our first observation is
Proposition \ref{inj}: when $B$ is smooth, the map
\begin{equation} \label{E:splitinj}
  \bigoplus_{i=0}^{d_X-d_B} h^{d_X -d_B -i}
\circ f^* \ : \ \bigoplus_{i=0}^{d_X-d_B} \CH_{l-i}(B) \longrightarrow
\CH_l(X)
\end{equation}
is injective for all $l$ and a left-inverse can be expressed as a
combination of the proper pushforward $f_*$, the refined pullback
$f^*$ and intersection with $h$. It is then not too surprising that,
when both $B$ and $X$ are smooth projective, the morphism of Chow
motives
$$ \bigoplus_{i=0}^{d_X-d_B} h^{d_X -d_B -i} \circ f^* \ : \
\bigoplus_{i=0}^{d_X-d_B} \h(B)(i) \longrightarrow \h(X)$$ is split
injective; see Theorem \ref{T:mainrevisited}. By taking a
cohomological realisation, for instance by taking Betti cohomology if
$k \subseteq \C$, we thus obtain that the map $$
\bigoplus_{i=0}^{d_X-d_B} h^{d_X -d_B -i} \circ f^* \ : \
\bigoplus_{i=0}^{d_X-d_B} \mathrm{H}^{n-2i}(B, \Q) \longrightarrow
\mathrm{H}^n(X,\Q)$$ is split injective for all $n$ and thus realises
the left-hand side group as a sub-Hodge structure of the right-hand
side group. This observation can be considered as a natural
generalisation of the elementary fact that a smooth projective variety
$X$ of Picard rank $1$ does not admit a non-constant dominant map to a
smooth projective variety of smaller dimension.

Let now $\Omega$ be a universal domain containing $k$, that is an
algebraically closed field containing $k$ which has infinite
transcendence degree over its prime subfield. Let us assume that there
is an integer $n$ such that the fibres $X_b$ of $f$ over
$\Omega$-points $b$ of $B$ satisfy $\CH_{l}(X_b) = \Q$ for all $l <
n$. If $f$ is flat, then Theorem \ref{main} shows that
\eqref{E:splitinj} is surjective for all $l<n$. When $X$ and $B$ are
both smooth projective, we deduce in Theorem \ref{T:gen-intro-text} a
direct sum decomposition of the Chow motive of $X$ as
\begin{equation} \label{E:dec} \h(X) \cong \bigoplus_{i=0}^{d_X-d_B}
  \h(B)(i) \oplus M(n),
\end{equation}
where $M$ is isomorphic to a direct summand of the motive of some
smooth projective variety $Z$ of dimension $d_X - 2n$. This notably
applies when $X$ is a projective bundle over a smooth projective
variety $B$ to give the well-known isomorphism
$\bigoplus_{i=0}^{d_X-d_B} \h(B)(i) \stackrel{\simeq}{\longrightarrow}
\h(X)$. Such a morphism is usually shown to be an isomorphism by an
existence principle, namely Manin's identity principle. Here, we
actually exhibit an explicit inverse to that isomorphism. The same
arguments are used in Theorem \ref{T:blowup} to provide an explicit
inverse to the smooth blow-up formula for Chow groups.  More
interesting is the case when a smooth projective variety $X$ is fibred
by complete intersections of low degree. For instance, the
decomposition \eqref{E:dec} makes it possible in Corollary
\ref{C:mainquadric} to construct a Murre decomposition (see Definition
\ref{D:Murre}) for smooth projective varieties fibred by quadrics over
a surface, thereby generalising a result of del Angel and
M\"uller--Stach \cite{dAMS} where a Murre decomposition was
constructed for $3$-folds fibred by conics over a surface, and also
generalising a previous result \cite{VialCK} where, in particular, a
Murre decomposition was constructed for $4$-folds fibred by quadric
surfaces over a surface.  Another consequence of the decomposition
\eqref{E:dec} is that rational and numerical equivalence agree on
smooth projective varieties $X$ fibred by quadrics over a curve or a
surface defined over a finite field; see Corollary \ref{C:ratnum}.  It
should be mentioned that our approach bypasses the technique of
Gordon--Hanamura--Murre \cite{GHM2}, where Chow--K\"unneth
decompositions are constructed from relative Chow--K\"unneth
decompositions. In our case, we do not require the existence of a
relative Chow--K\"unneth decomposition, nor do we require $f$ to be
smooth away from finitely many points as is the case in \cite{GHM2}.
Finally, it should be noted, for instance if $X$ is complex smooth
projective fibred by quadrics, that \eqref{E:dec} actually computes
some of the Hodge numbers of $X$ without going through a detailed
analysis of the Leray--Serre spectral sequence. \medskip

More generally, we are interested in computing, in some sense, the
Chow groups of $X$ in terms of the Chow groups of $B$ and of the
fibres of $f$. Let us first clarify what is meant by ``fibres''. We
observed in \cite[Theorem 1.3]{VialCK} that if $B$ is smooth and if a
general fibre of $f$ has trivial Chow group of zero-cycles
(\emph{i.e.} if it is spanned by the class of a point), then $f_* :
\CH_0(X) \r \CH_0(B)$ is an isomorphism with inverse a rational
multiple of $h^{d_X-d_B}\circ f^*$. We thus see that, as far as
zero-cycles on the fibres are concerned, it is enough to consider only
the general fibre.  For that matter, we show in Proposition
\ref{general-generic} that, provided the ground field is a universal
domain, it is actually enough that a very general fibre have trivial
Chow group of zero-cycles.  However, if one is willing to deal with
positive-dimensional cycles, it is no longer possible to ignore the
Chow groups of some of the fibres. For instance, if $\X_Y \r X$ is a
smooth blow-up along a smooth center $Y \subseteq X$, then
$\CH_1(\X_Y)$ is isomorphic to $\CH_1(X) \oplus \CH_0(Y)$, although a
general fibre of $\X_Y \r X$ is reduced to a point and hence has
trivial $\CH_1$. One may argue that a smooth blow-up is not flat. Let
us however consider as in \cite{B-conic} a complex flat conic
fibration $f : X \r \P^2$, where $X$ is smooth projective. All fibres
$F$ of $f$ satisfy $\CH_0(F)=\Q$.  A smooth fibre $F$ of $f$ is
isomorphic to $\P^1$ and hence satisfies $\CH_1(F)=\Q$. A singular
fibre $F$ of $f$ is either a double line, or the union of two lines
meeting at a point. In the latter case $\CH_1(F)= \Q \oplus \Q$. This
reflects in $\CH_1(X)$ and, as shown in \cite{B-conic},
$\CH_1(X)_{\hom}$ is isomorphic to the Prym variety attached to the
discriminant curve of $f$. This suggests that a careful analysis of
the degenerations of $f$ is required in order to derive some precise
information on the Chow groups of $X$.  On another perspective, the
following examples show what kind of limitation is to be expected when
dealing with fibres with non-trivial Chow groups. Consider first a
complex Enriques surface $S$ and let $T \r S$ be the $2$-covering by a
$K3$-surface $T$. We know that $\CH_0(S)=\Q$, and the fibres $T_s$ are
disjoint union of two points and hence satisfy $\CH_0(T_s)= \Q \oplus
\Q$. However, we cannot expect $\CH_0(T)$ to be correlated in some way
to the Chow groups of $S$ and of the fibres, because a theorem of
Mumford \cite{Mumford} says that $\CH_0(T)$ is infinite-dimensional in
a precise sense.  Another example is given by taking a pencil of
high-degree hypersurfaces in $\P^n$. Assume that the base locus $Z$ is
smooth. By blowing up $Z$, we get a morphism $\widetilde{\P}^n_Z \r
\P^1$.  In that case, $\CH_0(\P^1)=\Q$ and the $\CH_0$ of the fibres
is infinite-dimensional, but $\CH_0(\widetilde{\P}^n_Z) = \Q$.

Going back to the case where the general fibre of $f: X \r B$ has
trivial Chow group of zero-cycles, we see that $\CH_0(X)$ is supported
on a linear section of dimension $d_B$. We say that $\CH_0(X)$ has
niveau $d_B$. More generally, Laterveer \cite{Laterveer} defines a
notion of niveau on Chow groups as follows.  For a variety $X$, the
group $\CH_i(X)$ is said to have \emph{niveau} $\leq n$ if there
exists a closed subscheme $Z$ of $X$ of dimension $\leq i+n$ such that
$\CH_i(Z) \r \CH_i(X)$ is surjective, in other words if the $i$-cycles
on $X$ are supported in dimension $i+n$.  It can be proved
\cite[Theorem 1.7]{VialCK} that if a general fibre $F$ of $f : X \r B$
is such that $\CH_0(F)$ has niveau $\leq 1$, then $\CH_0(X)$ has
niveau $\leq d_B +1$. In that context, a somewhat more precise
question is: what can be said about the niveau of the Chow groups of
$X$ in terms of the niveau of the Chow groups of the fibres of $f : X
\r S$?  A statement one would hope for is the following: if the fibres
$X_b$ of $f : X \r B$ are such that $\CH_*(X_b)$ has niveau $\leq n$
for all $\Omega$-points $b \in B$, then $\CH_*(X)$ has niveau $\leq n
+ d_B$.  We cannot prove such a general statement but we prove it when
some of the Chow groups of the fibres of $f$ are either spanned by
linear sections or have niveau $0$, \emph{i.e.} when they are
finite-dimensional $\Q$-vector spaces.  Precisely, if $f : X \r B$ is
a complex projective surjective morphism onto a smooth
quasi-projective variety $B$, we show that $\CH_l(X)$ has niveau $\leq
d_B$ in the following cases:
\begin{itemize}
\item $\CH_i(X_b) = \Q$ for all $i\leq l$ and all $b \in B(\C)$
  (Theorem \ref{linear-theorem-complex});
\item $d_B = 1$ and $\CH_i(X_b)$ is finitely generated for all $i\leq
  l$ and all $b \in B(\C)$ (Theorem \ref{1-complex});
\item $f$ is smooth away from finitely many points, $\CH_i(X_b) = \Q$
  for all $i<l$ and $\CH_l(X_b)$ is finitely generated, for all $b \in
  B(\C)$ (Theorem \ref{genth-complex}).
\end{itemize}
These results, which are presented in Section \ref{general},
complement the generalisation of the projective bundle formula of
Theorem \ref{main} by dropping the flatness condition on $f$ and by
requiring in some cases that the Chow groups of the fibres be finitely
generated instead of one-dimensional. Their proofs use standard
techniques such as localisation for Chow groups (for that matter,
information on the Chow groups of the fibres of $f$ is extracted from
information on the Chow groups of the closed fibres of $f$ in Section
\ref{complex}), relative Hilbert schemes and a Baire category
argument.  Let us mention that the assumption of Theorem
\ref{genth-complex} on the singular locus of $f$ being finite is also
required in \cite{GHM2} where the construction of relative
Chow--K\"unneth decompositions is considered.  Finally, Theorem
\ref{T:niveau} gathers known results about smooth projective varieties
whose Chow groups have small niveau. Together with the results above,
in Section \ref{examples}, we prove some conjectures on algebraic
cycles (such as Kimura's finite-dimensionality conjecture
\cite{Kimura}, Murre's conjectures \cite{Murre1}, Grothendieck's
standard conjectures \cite{Kleiman}, the Hodge conjecture) for some
smooth projective varieties fibred by very low degree complete
intersection, or by cellular varieties over surfaces. For instance, we
show the existence of a Murre decomposition for smooth projective
varieties fibred by cellular varieties over a curve (Proposition
\ref{thcellular}) and for $6$-folds fibred by cubics over a curve
(Proposition \ref{cubics}), and the standard conjectures for varieties
fibred by smooth cellular varieties of dimension $\leq 4$ (Proposition
\ref{thcellular}) or by quadrics (Proposition \ref{quadrics1}) over a
surface.

\subsection*{Notations.}
We work over a field $k$ and we let $\Omega$ be a universal domain
that contains $k$. A \emph{variety} over $k$ is a reduced scheme of
finite type over $k$.  Throughout, $f : X \r B$ denotes a projective
surjective morphism defined over $k$ from a quasi-projective variety
$X$ of dimension $d_X$ to an irreducible quasi-projective variety $B$
of dimension $d_B$. Given a scheme $X$ over $k$, the group $\CH_i(X)$
is the $\Q$-vector space with basis the $i$-dimensional irreducible
reduced subschemes of $X$ modulo rational equivalence. By definition,
we set $\CH_j(X)=0$ for $j<0$ and we say that $\CH_i(X)$ is
\emph{finitely generated} if it is finitely generated as a $\Q$-vector
space, \emph{i.e.}  if it is a finite-dimensional $\Q$-vector
space. If $Z$ is an irreducible closed subscheme of $X$, we write
$[Z]$ for the class of $Z$ in $\CH_*(X)$. If $\alpha$ is the class of
a cycle in $\CH_*(X)$, we write $|\alpha|$ for the support in $X$ of a
cycle representing $\alpha$. If $Y$ is a scheme over $k$ and if
$\beta$ is a cycle in $\CH_*(X \times Y)$, we define its transpose
${}^t \beta \in \CH_*(Y \times X)$ to be the proper pushforward of
$\beta$ under the obvious map $\tau : X \times Y \r Y \times X$. If
$X$ and $Y$ are smooth projective, a cycle $\gamma \in \CH_*(X \times
Y)$ is called a correspondence. The correspondence $\gamma$ acts both
on $\CH_*(X)$ and $\CH_*(Y)$ in the following way. Let $p_X : X \times
Y \r X$ and $p_Y : X \times Y \r Y$ be the first and second
projections, respectively.  These are proper and flat and we may
define, for $\alpha \in \CH_*(X)$, $\gamma_* \alpha := (p_Y)_*(\gamma
\cdot p_X^*\alpha)$.  Here ``$\cdot$'' is the intersection product on
non-singular varieties as defined in \cite[\S8]{Fulton}. We then
define, for $\beta \in \CH_*(Y)$, $\gamma^*\beta :=
({}^t\gamma)_*\beta$. Given another smooth projective variety $Z$ and
a correspondence $\gamma' \in \CH_*(Y \times Z)$, the composite
$\gamma' \circ \gamma \in \CH_*(X \times Z)$ is defined to be
$(p_{XZ})_*(p_{XY}^*\gamma \cdot p_{YZ}^*\gamma')$, where $p_{XY} : X
\times Y \times Z \r X \times Y$ is the projection and likewise for
$p_{XZ}$ and $p_{YZ}$. The composition of correspondence is compatible
with the action of correspondences on Chow groups \cite[\S
16]{Fulton}.

Motives are defined in a covariant setting and the notations are those
of \cite{VialCK}.  Briefly, a Chow motive (or motive, for short) $M$
is a triple $(X,p,n)$ where $X$ is a variety of pure dimension $d_X$,
$p \in \CH_d(X\times X)$ is an idempotent ($p\circ p = p$) and $n$ is
an integer. The motive of $X$ is denoted $\h(X)$ and, by definition,
is the motive $(X,\Delta_X,0)$ where $\Delta_X$ is the class in
$\CH_{d_X}(X \times X)$ of the diagonal in $X \times X$. We write
$\mathds{1}$ for the unit motive $(\Spec k, \Delta_{\Spec k},0) =
\h(\Spec k)$.  With our covariant setting, we have $\h(\P^1) =
\mathds{1} \oplus \mathds{1}(1)$.  A morphism between two motives
$(X,p,n)$ and $(Y,q,m)$ is a correspondence in $q \circ \CH_{d_X+n-m}
(X \times Y) \circ p$. If $f : X \r Y$ is a morphism, $\Gamma_f$
denotes the graph of $f$ in $X \times Y$. By abuse, we also write
$\Gamma_f \in \CH_{d_X}(X \times Y)$ for the class of the graph of
$f$. It defines a morphism $\Gamma_f : \h(X) \r \h(Y)$.  By definition
we have $\CH_i(X,p,n) = p_*\CH_{i-n}(X)$ and $\mathrm{H}_i(X,p,n) =
p_*\mathrm{H}_{i-2n}(X)$, where we write $\mathrm{H}_i(X) :=
\mathrm{H}^{2d-i}(X(\C),\Q)$ for singular homology when $k \subseteq
\C$, or $\mathrm{H}_i(X) :=
\mathrm{H}^{2d-i}(X_{\overline{k}},\Q_\ell)$ for $\ell$-adic homology
($\ell \neq \mathrm{char} \, k$) otherwise.

Given an irreducible scheme $Y$ over $k$, $\eta_Y$ denotes the generic
point of $Y$. If $f : X \r B$ and if $Y$ is a closed irreducible
subscheme of $B$, $X_{\eta_Y}$ denotes the fibre of $f$ over the
generic point of $Y$ and $X_{\overline{\eta}_Y}$ denotes the fibre of
$f$ over a geometric generic point of $Y$

\section{Surjective morphisms and motives}

Let us start by recalling a few facts about intersection theory. Let
$f : X \r Y$ be a morphism of schemes defined over $k$ and let $l$ be
an integer. If $f$ is proper, then there is a well-defined proper
pushforward map $f_* : \CH_l(X) \r \CH_l(Y)$; see
\cite[\S1.4]{Fulton}. If $f$ is flat, then there is a well-defined
flat pullback map $f^* : \CH^l(Y) \r \CH^l(X)$; see
\cite[\S1.7]{Fulton}. Pullbacks can also be defined in the following
two situations.  On the one hand, if $D$ is a Cartier divisor with
support $\iota : |D| \hookrightarrow X$, there is a well-defined Gysin
map $\iota^* : \CH_l(X) \r \CH_{l-1}(|D|)$ and the composite $\iota_*
\circ \iota^* : \CH_l(X) \r \CH_{l-1}(X)$ does not depend on the
linear equivalence class of $D$, that is, there is a well-defined
action of the Picard group $\mathrm{Pic}(X)$ on $\CH_l(X)$; see
\cite[\S2]{Fulton}. For instance, if $X$ is a quasi-projective variety
given with a fixed embedding $X \hookrightarrow \P^N$, then there is a
well-defined map $h : \CH_l(X) \r \CH_{l-1}(X)$ given by intersecting
with a hyperplane section of $X$. More generally, if $\tau : Y
\hookrightarrow X$ is a locally complete intersection of codimension
$r$, then there is a well-defined Gysin map $\tau^* : \CH_l(X)
\rightarrow \CH_{l-r}(Y)$; see \cite[\S 6]{Fulton}.  For $n \geq 0$,
we write $h^0 = \id : \CH_l(X) \r \CH_l(X)$ and $h^n$ for the $n$-fold
composite $h \circ \ldots \circ h : \CH_l(X) \r \CH_{l-n}(X)$. By
functoriality of Gysin maps \cite[\S 6.5]{Fulton}, if $\iota^n : H^n
\hookrightarrow X$ denotes a linear section of codimension $n$, then
the composite map $h \circ \ldots \circ h$ coincides with $\iota^n_*
\circ (\iota^n)^*$.  When $X$ is smooth projective, we write
$\Delta_{H^n}$ for the diagonal inside $H^n \times H^n$, and the
correspondence $\Gamma_{\iota^n} \circ {}^t\Gamma_{\iota^n} = (\iota^n
\times \iota^n)_*[\Delta_{H^n}] \in \CH_{d_X-n}(X \times X)$ induces a
map $\CH_l(X) \r \CH_{l-n}(X)$ that coincides with the map $h^n$; see
\cite[\S16]{Fulton}. By abuse, we also write $h^n = \Gamma_{\iota^n}
\circ {}^t\Gamma_{\iota^n}$ for $n >0$ and $h^0 := [\Delta_X]$. On the
other hand, if $f : X \r Y$ is a morphism to a non-singular variety
$Y$ and if $x \in \CH_*(X)$ and $y \in \CH_*(Y)$, then there is a
well-defined refined intersection product $x \cdot_f y \in \CH_*(|x|
\cap f^{-1}(|y|))$, where ``$\cap$'' denotes the scheme-theoretic
intersection; see \cite[\S 8]{Fulton}. The pullback $f^*y$ is then
defined to be the proper pushforward of $[X] \cdot_f y$ in
$\CH_*(X)$. Let us denote $\gamma_f : X \rightarrow X \times Y$ the
morphism $x \mapsto (x,f(x))$. Because $Y$ is non-singular, this
morphism is a locally complete intersection morphism and the pullback
$f^*$ is by definition $\gamma_f^* \circ p_Y^*$, where $p_Y: X \times
Y \rightarrow Y$ is the projection and $\gamma_f^*$ is the Gysin map;
see \cite[\S 8]{Fulton}. Finally, if $f$ is flat, then this pullback
map coincides with flat pullback \cite[Prop.  8.1.2]{Fulton}.\medskip

We have the following basic lemma.

\begin{lemma}
  \label{L:dominant} Let $f : X \r B$ be a projective surjective
  morphism between two quasi-projective varieties. Let $X'
  \hookrightarrow X$ be a linear section of $X$ of dimension $\geq
  d_B$. Then $f|_{X'} : X' \r B$ is surjective.
\end{lemma}
\begin{proof} Let $X \hookrightarrow \P^N$ be an embedding of $X$ in
  projective space and let $H \hookrightarrow \P^N$ be a linear
  subspace such that $X'$ is obtained as the pullback of $X$ along $H
  \hookrightarrow \P^N$. The linear subvariety $H$ has codimension at
  most $d_X-d_B$ in $\P^N$ while a geometric fibre of $f$ has
  dimension at least $d_X - d_B$. Thus every geometric fibre of $f$
  meets $H$ and hence $X'$. It follows that $f|_{X'}$ is surjective.
\end{proof}

\begin{lemma} \label{dominant} 
  Let $f : X \r B$ be a projective surjective morphism to a smooth
  quasi-projective variety $B$. Then there exists a positive integer
  $n$ such that, for all $i$, $f_* \circ h^{d_X-d_B} \circ f^* :
  \CH_i(B) \r \CH_i(B)$ is multiplication by $n$. If moreover $X$ is
  smooth and $B$ is projective, then
  $$\Gamma_f \circ h^{d_X -d_B} \circ {}^t\Gamma_f = n \cdot \Delta_B
  \in \CH_{d_B}(B \times B).$$
\end{lemma}
\begin{proof} Let $\iota' := \iota^{d_X-d_B} : H' \hookrightarrow X$
  be a linear section of $X$ of dimension $d_B$. We first check, for
  lack of reference, that $(f \circ \iota')^* = (\iota')^* \circ f^*$
  on Chow groups. Here, $f \circ \iota'$ and $f$ are morphisms to the
  non-singular variety $B$ and as such the pullbacks $(f \circ
  \iota')^*$ and $(\iota')^*$ are the ones of \cite[\S 8]{Fulton},
  while $\iota'$ is the inclusion of a locally complete intersection
  and as such the pullback $(\iota')^*$ is the Gysin pullback of
  \cite[\S 6]{Fulton}. Let $\sigma \in \CH^*(B)$, then
  $(\iota')^*f^*\sigma = (\iota')^*\gamma_f^*([X]\times \sigma) =
  (\gamma_f \circ \iota')^*([X]\times \sigma)$, where the second
  equality follows from the functoriality of Gysin maps \cite[\S
  6.5]{Fulton}. Since $\gamma_f \circ \iota' = (\iota'\times
  \mathrm{id}_B) \circ \gamma_{f\circ \iota'}$, we get by using
  functoriality of Gysin maps once more that $(\iota')^*f^*\sigma =
  \gamma_{f\circ \iota'}^*(\iota' \times \mathrm{id}_B)^*([X] \times
  \sigma)$. Now, we have $(\iota' \times \mathrm{id}_B)^*([X]\times
  \sigma) = (\iota')^*[X]\times \sigma = [H'] \times \sigma$; see
  \cite[Example 6.5.2]{Fulton}. We therefore obtain that
  $(\iota')^*f^*\sigma = \gamma_{f\circ \iota'}^*([H'] \times \sigma)
  := (f\circ \iota')^*\sigma$, as claimed.
    
  Thus, since in addition both $f$ and $\iota'$ are proper, we have by
  functoriality of proper pushforward $(f \circ \iota')_* (f \circ
  \iota')^* = f_* \iota'_* (\iota')^* f^* = f_* \circ h^{d_X-d_B}
  \circ f^*$. By Lemma \ref{L:dominant}, the composite morphism $g :=
  f \circ \iota'$ is generically finite, of degree $n$ say. It follows
  from the projection formula \cite[Prop. 8.1.1(c)]{Fulton} and from
  the definition of proper pushforward that, for all $\gamma \in
  \CH_i(B)$, $$f_* \circ h^{d_X-d_B} \circ f^*\gamma = g_* ( [H']
  \cdot g^*\gamma) = g_*([H']) \cdot \gamma = n [B] \cdot \gamma = n
  \gamma.$$

  Assume now that $X$ and $B$ are smooth projective. In that case, we
  have $\Gamma_f \circ h^{d_X -d_B} \circ {}^t\Gamma_f = \Gamma_g
  \circ {}^t\Gamma_g : = (p_{1,3})_*(p_{1,2}^*{}^t\Gamma_g \cdot
  p_{2,3}^*\Gamma_g)$, where $p_{i,j}$ denotes projection from $B
  \times H' \times B$ to the $(i,j)$-th factor. By refined
  intersection, we see that $\Gamma_g \circ {}^t\Gamma_g$ is supported
  on $(p_{1,3})([{}^t\Gamma_g \times B] \cap [B \times \Gamma_g])$,
  which itself is supported on the diagonal of $B \times B$. Thus
  $\Gamma_f \circ h^{d_X -d_B} \circ {}^t\Gamma_f$ is a multiple of
  $\Delta_B$.  We have already showed that $f_* \circ h^{d_X-d_B}
  \circ f^* =(\Gamma_f \circ h^{d_X -d_B} \circ {}^t\Gamma_f)_*$ acts
  by multiplication by $n$ on $\CH_i(B)$. Therefore, $\Gamma_f \circ
  h^{d_X -d_B} \circ {}^t\Gamma_f = n \cdot \Delta_B$.
\end{proof}

The following lemma is reminiscent of \cite[Prop. 3.1.(a)]{Fulton}.

\begin{lemma} \label{orth} Let $f : X \r B$ be a projective morphism
  to a smooth quasi-projective variety $B$. Then, for all $i$, $f_*
  \circ h^{l} \circ f^* : \CH_i(B) \r \CH_{i+d_X-d_B-l}(B)$ is the
  zero map for all $l < d_X - d_B$. If moreover $X$ is smooth and $B$
  is projective, then
  \begin{center}
    $\Gamma_f \circ h^l \circ {}^t\Gamma_f = 0 \in \CH_{d_X-l}(B
    \times B)$ for all $l < d_X - d_B$.
  \end{center}
 \end{lemma}
 \begin{proof} By refined intersection \cite[\S 8]{Fulton}, the
   pullback $f^*\alpha$ is represented by a well-defined class in
   $\CH_{i + d_X - d_B}(f^{-1}(|\alpha|))$ for any cycle $\alpha \in
   \CH_i(B)$. It follows that $h^{l} \circ f^* \alpha$ is represented
   by a well-defined class in $\CH_{i + d_X - d_B
     -l}(f^{-1}(|\alpha|))$.  Since $f|_{f^{-1}(|\alpha|)} :
   f^{-1}(|\alpha|) \r |\alpha|$ is proper, we see by proper
   pushforward that $f_* \circ h^{l} \circ f^* \alpha$ is represented
   by a well-defined cycle $\beta \in \CH_{i+d_X-d_B-l}(|\alpha|)$.
   But then, $\dim |\alpha| = i$ so that if $l < d_X - d_B$, then
   $\CH_{i+d_X-d_B-l}(|\alpha|)=0$.

   Let us now assume that $X$ and $B$ are smooth
   projective. Let $\iota^l :
   H^l \hookrightarrow X$ be a linear section of $X$ of codimension
   $l$, and let $h^l$ be the class of $(\iota^l \times \iota^l)(
   \Delta_{H^l})$ in $\CH_{d_X-l}(X \times X)$.  By definition we have
   $\Gamma_f \circ h^l \circ {}^t\Gamma_f =
   (p_{1,4})_*(p_{1,2}^*{}^t\Gamma_f \cdot p_{2,3}^*h^l \cdot
   p_{3,4}^*\Gamma_f)$, where $p_{i,j}$ denotes projection from $B
   \times X \times X \times B$ to the $(i,j)$-th factor. These
   projections are flat morphisms, therefore by flat pullback we have
   $p_{1,2}^*{}^t\Gamma_f = [{}^t\Gamma_f \times X \times B]$,
   $p_{2,3}^*h^l = [B \times \Delta_{H^l} \times B]$ and
   $p_{3,4}^*\Gamma_f = [B \times X \times \Gamma_f]$. By refined
   intersection, the intersection of the closed subschemes
   ${}^t\Gamma_f \times X \times B$, $B \times \Delta_{H^l} \times B$
   and $B \times X \times \Gamma_f$ of $B \times X \times X \times B$
   defines a $(d_X-l)$-dimensional class supported on their
   scheme-theoretic intersection $\{(f(h),h,h,f(h)) : h \in H^l \}
   \subset B \times X \times X \times B$. Since $f$ is projective,
   this is a closed subset of dimension $d_X-l$. Also its image under
   the projection $p_{1,4}$ has dimension at most $d_B$, which is
   strictly less than $d_X -l$ by the assumption made on $l$. The
   projection $p_{1,4}$ is a proper map and hence, by proper
   pushforward, we get that $(p_{1,4})_* [\{(f(h),h,h,f(h)) \in B
   \times X \times X \times B : h \in H^l \}] =0$.
 \end{proof}

 \begin{theorem} \label{T:mainrevisited} Let $f : X \r B$ be a
   surjective morphism of smooth projective varieties over $k$.
   Consider the following two morphisms of motives
  $$\Phi := \bigoplus_{i=0}^{d_X-d_B} h^{d_X -d_B -i} \circ {}^t\Gamma_f \ : \
  \bigoplus_{i=0}^{d_X-d_B} \h(B)(i) \longrightarrow \h(X)$$ and
  $$\Psi :=   \bigoplus_{i=0}^{d_X-d_B} \Gamma_f  \circ h^{i} \ : \
  \h(X) \longrightarrow \bigoplus_{i=0}^{d_X-d_B} \h(B)(i).$$ Then
  $\Psi \circ \Phi$ is an automorphism.
 \end{theorem}
 \begin{proof} 
   The endomorphism $\Psi \circ \Phi : \bigoplus_{i=0}^{d_X-d_B}
   \h(B)(i) \r \bigoplus_{i=0}^{d_X-d_B} \h(B)(i)$ can be represented
   by the $(d_X-d_B+1) \times (d_X-d_B+1)$-matrix whose
   $(i,j)^\mathrm{th}$-entries are the morphisms $$(\Psi \circ
   \Phi)_{i,j} = \Gamma_f \circ h^{d_X-d_B-(j-i)} \circ {}^t\Gamma_f \
   : \ \h(B)(j-1) \r \h(B)(i-1).$$ By Lemma \ref{dominant}, there is a
   non-zero integer $n$ such that the diagonal entries satisfy $(\Psi
   \circ \Phi)_{i,i} = n \cdot \id_{\h(B)(i-1)}$. By Lemma \ref{orth},
   $(\Psi \circ \Phi)_{i,j} = 0$ as soon as $j>i$. Therefore, $$N :=
   \, \id - \frac{1}{n} \cdot \Psi \circ \Phi$$ is a nilpotent
   endomorphism of $\bigoplus_{i=0}^{d_X-d_B} \h(B)(i) $ with
   $N^{d_X-d_B+1} = 0$. Let us define
   \begin{eqnarray*} \Xi & := & (n\cdot \id - n\cdot N)^{-1} \\ & = &
     \frac{1}{n} \cdot \big(\id + N + N^2 + \cdot + N^{d_X-d_B}\big).
\end{eqnarray*}
It then follows that $\Xi$ is the inverse of $\Psi \circ \Phi$.
\end{proof}
 In the situation of Theorem \ref{T:mainrevisited}, the morphism
$$\Theta := \Xi \circ \Psi$$ then defines a left-inverse to
$\Phi$ and the endomorphism $$p:= \Phi \circ \Theta = \Phi \circ \Xi
\circ \Psi \in \End(\h(X))$$ is an idempotent.

\begin{proposition} \label{P:selfdualidempotent} With the notations
  above, the idempotent $p \in \End(\h(X)) = \CH_{d_X}(X \times X)$
  satisfies $p = {}^tp$. Moreover, the morphism $\Psi \circ p : (X,p)
  \r \bigoplus_{i=0}^{d_X-d_B} \h(B)(i)$ is an isomorphism with
  inverse $p \circ \Phi \circ \Xi$.
\end{proposition}
\begin{proof}
  The second claim consists of the following identities: $\Psi \circ p
  \circ p \circ \Phi \circ \Xi = \Psi \circ p \circ \Phi \circ \Xi =
  \Psi \circ \Phi \circ \Xi \circ \Psi \circ \Phi \circ \Xi = \id
  \circ \id = \id$ and $p \circ \Phi \circ \Xi \circ \Psi \circ p = p
  \circ p \circ p = p$.

  As for the first claim, we have $$p = \frac{1}{n} \cdot \Phi \circ
  \big(1+N+\ldots +N^{d_X-d_B}\big) \circ \Psi.$$ Recall that $N = \id
  - \frac{1}{n}\cdot \Psi \circ \Phi$, so that it is enough to see
  that ${}^t(\Phi \circ \Psi) = \Phi \circ \Psi$. A straightforward
  computation gives $$\Phi \circ \Psi = \sum_{i=0}^{d_X-d_B}
  h^{d_X-d_B-i} \circ {}^t\Gamma_f \circ \Gamma_f \circ h^i.$$ We may
  then conclude by noting that the correspondence $h \in \CH_{d_X-1}(X
  \times X)$ satisfies $h = {}^th$.
\end{proof}

Finally, let us conclude with the following counterpart of Theorem
\ref{T:mainrevisited} that deals with the Chow groups of
quasi-projective varieties.

 \begin{proposition} \label{inj} Let $f : X \r B$ be a projective
   surjective morphism to a smooth quasi-projective variety $B$.  Then
   the map
  $$\Phi_* =  \bigoplus_{i=0}^{d_X-d_B} h^{d_X -d_B -i} \circ f^* \ : \
  \bigoplus_{i=0}^{d_X-d_B} \CH_{l-i}(B) \longrightarrow \CH_l(X)$$ is
  split injective and its left-inverse is a polynomial function in
  $f_*, f^*$ and $h$.
\end{proposition}

\begin{proof} Thanks to Lemma \ref{dominant} and to Lemma \ref{orth},
  there is a non-zero integer $n$ such that
  $$f_* \circ h^i \circ f^* : \CH_l(B) \r \CH_{l+d_X-d_B
    - i}(B)$$ is multiplication by  $n$ if $i=d_X-d_B$
  and is zero if $i<d_X - d_B$.

  Let us write $\Psi_*$ for $ \bigoplus_{j=0}^{d_X-d_B} f_* \circ h^j :
  \CH_l(X) \rightarrow \bigoplus_{j=0}^{d_X-d_B} \CH_{l-j}(B)$.  In
  order to prove the injectivity of $\Phi_*$, it suffices to show that
  the composite
  $$\Psi_* \circ \Phi_* \ : \
  \bigoplus_{i=0}^{d_X-d_B} \CH_{l-i}(B) \longrightarrow \CH_l(X)
  \longrightarrow \bigoplus_{j=0}^{d_X-d_B} \CH_{l-j}(B)$$ is an
  isomorphism. But then, as in the proof of Theorem
  \ref{T:mainrevisited}, we see that $\Psi_* \circ \Phi_*$ can be
  represented by a lower triangular matrix whose diagonal entries'
  action on $\CH_{l-i}(B)$ is given by multiplication by $n$.
\end{proof}

\section{On the Chow groups of the fibres} \label{complex}

In this section, we fix a universal domain $\Omega$.  The following
statement was communicated to me by Burt Totaro.

\begin{lemma} \label{lemma} Let $f : X \r B$ be a morphism of
  varieties over $\Omega$ and let $F$ be a geometric generic fibre of
  $f$. Then there is a subset $U \subseteq B(\Omega)$ which is a
  countable intersection of nonempty Zariski open subsets such that
  for each point $b \in U$, there is an isomorphism from the field
  $\Omega$ to the field $\overline{\Omega(B)}$ such that this
  isomorphism turns the scheme $X_b$ over $\Omega$ into the scheme $F$
  over $\overline{\Omega(B)}$. In other words, a very general fibre of
  $f$ is isomorphic to $F$ as an abstract scheme.
  
  Consequently, for each point $p \in U$, $\CH_i(X_b)$ is isomorphic to
    $\CH_i(F)$ for all $i$.
\end{lemma}
\begin{proof}
  There exist a countable subfield $K \subset \Omega$ and varieties
  $X_0$ and $B_0$ defined over $K$ together with a $K$-morphism $f_0 :
  X_0 \r B_0$ such that $f = f_0 \times_{\Spec K} \Spec \Omega$. Let
  us define $U \subseteq B(\Omega)$ to be $\bigcap_{Z_0}
  (B_0\backslash Z_0)_\Omega (\Omega)$, where the intersection runs
  through all proper $K$-subschemes $Z_0$ of $B_0$. Note that there
  are only countably many such subschemes of $B_0$ and that $U$ is the
  set of $\Omega$-points of $B = B_0 \times_{\Spec K} \Spec \Omega$
  that do not lie above a proper Zariski-closed subset of $B_0$.
      
  Let now $b : \Spec \Omega \r B$ be a $\Omega$-point of $B$ that lies
  in $U$, \emph{i.e.} such that the composite map $\beta : \Spec
  \Omega \stackrel{b}{\longrightarrow} B \r B_0$ is dominant, or
  equivalently such that the composite map $\beta$ factors as
  $\eta_{B_0} \circ \alpha$ for some morphism $\alpha : \Spec \Omega
  \rightarrow \Spec K(B_0)$, where $\eta_{B_0} : \Spec K(B_0) \r B_0$
  is the generic point of $B_0$. Since $X$ is pulled back from $X_0$
  along $B \rightarrow B_0$, we see that $X_b$ the fibre of $f$ at $b$
  is the pull back of the generic fibre $(X_0)_{\eta_{B_0}}$ along
  $\alpha$.
  Consider then $ \overline{\eta_B} : \Spec \overline{\Omega(B)} \r B$
  a geometric generic point of $B$ such that $X_{ \overline{\eta_B}} =
  F$. Since the composite map $\Spec \overline{\Omega(B)} \r B \r B_0$
  factors through $\eta_{B_0} : \Spec K(B_0) \r B_0$, we see as before
  that $F$ is the pull-back of the generic fibre $(X_0)_{\eta_{B_0}}$
  along some morphism $\alpha' : \Spec \overline{\Omega(B)}
  \rightarrow \Spec K(B_0)$. The fields $\overline{\Omega(B)}$ and
  $\Omega$ are algebraically closed fields of infinite transcendence
  degree over $ K(B_0)$ and there thus exists an isomorphism
  $\overline{\Omega(B)} \cong \Omega$ fixing $K(B_0)$. Hence, the
  fibre $X_b$ identifies with $F$ after pullback by the isomorphism
  $\Spec \Omega \cong \Spec \overline{\Omega(B)}$ over $\Spec K(B_0)$.
  
  The last statement follows from the fact that the Chow groups of a
  variety $X$ over a field only depend on $X$ as a scheme. Precisely,
  if one denotes $\psi_b : X_b \rightarrow F$ an isomorphism of
  schemes, then the proper pushforward map $(\psi_b)_* : \CH_i(X_b)
  \rightarrow \CH_i(F)$ is an isomorphism with inverse
  $(\psi_b^{-1})_* : \CH_i(F) \rightarrow \CH_i(X_b)$.
\end{proof}

The following lemma will be useful to refer to.

\begin{lemma} \label{L:universal} Let $f : X \r B$ be a projective
  surjective morphism defined over $\Omega$ onto a quasi-projective
  variety $B$.  Assume that $\CH_l(X_b) = \Q$ (resp. $\CH_l(X_b)$ is
  finitely generated) for all $b \in B(\Omega)$. Then
  $\CH_l(X_{\eta_D}) = \Q$ (resp. $\CH_l(X_{\eta_D})$ is finitely
  generated) for all irreducible subvarieties $D$ of $X$.
\end{lemma}
\begin{proof} Let $D$ be an irreducible subvariety of $B$ and let
  $\overline{\eta}_D \r D$ be a geometric generic point of $D$. By
  Lemma \ref{lemma} applied to $X_D := X \times_B D$, there is a
  closed point $d \in D$ such that $\CH_l(X_{\overline{\eta}_D})$ is
  isomorphic to $\CH_l(X_d)$.  By assumption $\CH_l(X_d)=\Q$ (resp.
  $\CH_l(X_d)$ is finitely generated).  Therefore
  $\CH_l(X_{\overline{\eta}_D})=\Q$ (resp.
  $\CH_l(X_{\overline{\eta}_D})$ is finitely generated), too. By a
  norm argument for Chow groups, the pullback map $\CH_l(X_{{\eta}_D})
  \r \CH_l(X_{\overline{\eta}_D})$ is injective.  Hence
  $\CH_l(X_{{\eta}_D})=\Q$ (resp.  $\CH_l(X_{{\eta}_D})$ is finitely
  generated).
\end{proof}

The following definition is taken from Laterveer \cite{Laterveer}.

\begin{definition} \label{D:niveau} Let $X$ be a variety over $k$. The
  Chow group $\CH_i(X)$ is said to have \emph{niveau} $\leq r$ if
  there exists a closed subscheme $Y \subset X$ of dimension $i+r$
  such that the proper pushforward map $\CH_i(Y_\Omega) \r
  \CH_i(X_\Omega)$ is surjective.
\end{definition}

\begin{proposition} \label{general-generic} Let $f : X \rightarrow B$
  be a generically smooth, projective and dominant morphism onto a
  smooth quasi-projective variety $B$ defined over $\Omega$. Let $n$
  be a non-negative integer. The
  following statements are equivalent.
  \begin{enumerate}
  \item If $F$ is a general fibre, then $\CH_0(F)$ has niveau $\leq n$;
\item If $F$ is a very general fibre, then $\CH_0(F)$ has niveau $\leq n$;
\item If $F$ is a geometric generic fibre, then $\CH_0(F)$ has niveau
  $\leq n$.
  \end{enumerate}
\end{proposition}
\begin{proof}
  The implication $(1) \Rightarrow (2)$ is obvious. Let us prove $(2)
  \Rightarrow (3)$. Let $X_b$ be a very general fibre and $F$ a
  geometric generic fibre of $f$, and, by Lemma \ref{lemma}, let
  $\psi_b : X_b \rightarrow F$ be an isomorphism of schemes. Assume
  that there is a closed subscheme $Z$ of dimension $\leq r$ in $X_b$,
  for some integer $r$, such that the proper pushforward $\CH_0(Z)
  \rightarrow \CH_0(X_b)$ is surjective. Then, denoting $Z'$ the image
  of $Z$ in $F$ under $\psi_b$, functoriality of proper pushforwards
  implies that $\CH_0(Z') \rightarrow \CH_0(F)$ is surjective. We may
  then conclude by noting that the subscheme $Z'$ has dimension $\leq
  r$ in $F$.
 
  As for $(3) \Rightarrow (1)$, let $Y$ be a subvariety of $F$ defined
  over $\overline{\eta}_B$ such that $\CH_0(Y) \r \CH_0(F)$ is
  surjective. The technique of decomposition of the diagonal of
  Bloch--Srinivas \cite{BS} gives $\Delta_F = \Gamma_1 + \Gamma_2 \in
  \CH_{\dim F}(F \times F)$, where $\Gamma_1$ is supported on $F
  \times Y$ and $\Gamma_2$ is supported on $D\times F$ for some
  divisor $D$ in $F$. Consider a Galois extension $K/\Omega(B)$ over
  which the above decomposition and the morphism $Y \r F$ are defined,
  and consider an \'etale morphism $U \r B$ with $\Omega(U)=K$ such
  that $f$ restricted to $U$ is smooth. Let $u$ be a $\Omega$-point of
  $U$ and let $X_u$ be the fibre of $f$ over $u$.  Then the
  decomposition $\Delta_F = \Gamma_1 + \Gamma_2$ specialises \cite[\S
  20.3]{Fulton} on $X_u \times X_u$ to a similar decomposition, where
  $\Gamma_1|_{X_u \times X_u}$ is supported on $X_u \times Y_u$ and
  $\Gamma_2|_{X_u \times X_u}$ is supported on $D_u \times X_u$.
  Letting it act on zero-cycles, we see that $\CH_0(X_u)$ is supported
  on $Y_u$.
\end{proof}

\begin{remark}
  When $n=0$ or $n=1$, the statements of Proposition
  \ref{general-generic} are further equivalent to $\CH_0(F)$ having
  niveau $\leq n$ for $F$ the generic fibre of $f$. Indeed, if $X$ is
  a smooth projective variety such that $\CH_0(X)$ has niveau $\leq
  1$, then $\CH_0(X)$ is supported on a one-dimensional linear section
  \cite[Proposition 1.6]{Jannsen}. In particular, $\CH_0(X)$ is
  supported on a one-dimensional subvariety of $X$ which is defined
  over a field of definition of $X$. Note that, for general $n$, it is
  a consequence of the Lefschetz hyperplane theorem and of the
  Bloch--Beilinson conjectures that if $\CH_0(X)$ has niveau $\leq n$,
  then $\CH_0(X)$ is supported on an $n$-dimensional linear section of
  $X$.
\end{remark}

\section{A generalisation of the projective bundle formula for Chow
  groups} \label{hypsection}

We establish a formula that is analogous to the projective bundle
formula for Chow groups.
Our formula holds for flat morphisms, rather than Zariski
locally trivial morphisms as is the case for the projective bundle
formula.  However, since a flat morphism does not have any local
sections in general, it only holds with rational coefficients.

\begin{proposition} \label{P:surj} Let $f : X \r B$ be a flat projective
  surjective morphism of quasi-projective varieties.  Let $l \geq 0$
  be an integer.  Assume that $$\CH_{l-i}(X_{\eta_{B_i}})= \Q$$ for
  all $0 \leq i \leq \min (l,d_B)$ and for all closed irreducible
  subschemes $B_i$ of $B$ of dimension $i$, where $\eta_{B_i}$ is the
  generic point of $B_i$.

  Then the map
  $$\Phi_* =  \bigoplus_{i=0}^{d_X-d_B} h^{d_X -d_B -i} \circ f^* \ : \
  \bigoplus_{i=0}^{d_X-d_B} \CH_{l-i}(B) \longrightarrow \CH_l(X)$$ is
  surjective.
\end{proposition}
\begin{proof}
  The case when $d_B = 0$ is obvious. Let us proceed by induction on
  $d_B$. We have the localisation exact sequence $$\bigoplus_{D \in
    B^{1}} \CH_l(X_D) \longrightarrow \CH_l(X) \longrightarrow
  \CH_{l-d_B}(X_{\eta_B}) \longrightarrow 0,$$ where the direct sum is
  taken over all irreducible divisors of $B$. If $l \geq d_B$, let $Y$
   be a linear section of $X$ of dimension $l$.  By Lemma
  \ref{L:dominant}, $f|_Y : Y \r B$ is surjective. The restriction map
  $\CH_l(X) \rightarrow \CH_{l-d_B}(X_{\eta_B})$ is the direct limit
  of the flat pullback maps $\CH_l(X) \r \CH_l(X_U)$ taken over all
  open subsets $U$ of $B$; see \cite[Lemma 1A.1]{Bloch-lectures}.
  Therefore $\CH_l(X) \rightarrow \CH_{l-d_B}(X_{\eta_B})$ sends the
  class of $Y$ to the class of $Y_{\eta_B}$ inside
  $\CH_{l-d_B}(X_{\eta_B})$.  But then this class is non-zero because
  the restriction to $\eta_B$ of a linear section of $Y$ of dimension
  $d_B$ has positive degree.  Furthermore, if $[B]$ denotes the class
  of $B$ in $\CH_{d_B}(B)$, then the class of $Y$ is equal to $h^{d_X-
    l } \circ f^* [B]$ in $\CH_l(X)$.  Thus, since by assumption
  $\CH_{l-d_B}(X_{\eta_B}) = \Q$, the composite map
  $$\CH_{d_B}(B) \stackrel{h^{d_X-l} \circ f^*}{\longrightarrow} \CH_l(X)
  \r \CH_{l-d_B}(X_{\eta_B})$$ is surjective.

  Consider now the fibre square \begin{center} $ \xymatrix{ X_D
      \ar[d]_{f_D} \ar[r]^{j_D'} & X \ar[d]^{f} \\ D \ar[r]^{j_D} &
      B.}$
  \end{center} Then $f_D : X_D \r D$ is flat and its fibres above
  points of $D$ satisfy the assumptions of the theorem. Therefore, by
  the inductive assumption, we have a surjective map
  $$\bigoplus_{i=0}^{d_X-d_B} h^{d_X -d_B -i} \circ f_D^* \ : \
  \bigoplus_{i=0}^{d_X-d_B} \CH_{l-i}(D) \longrightarrow \CH_l(X_D).$$
  Furthermore, since $f$ is flat and $j_D$ is proper, we have the
  formula \cite[Prop. 1.7 \& Th. 6.2]{Fulton} $$j_{D*}' \circ h^{d_X
    -d_B -i} \circ f_D^* = h^{d_X -d_B -i} \circ f^* \circ j_{D*} \ :
  \ \CH_{l-i}(D) \r \CH_l(X).$$ Hence, the image of $\Phi_*$ contains
  the image of
  $$\bigoplus_{D \in B^1} \bigoplus_{i=0}^{d_X-d_B} j'_{D*} \circ
  h^{d_X -d_B -i} \circ f_D^* \ : \ \bigoplus_{D \in B^1}
  \bigoplus_{i=0}^{d_X-d_B} \CH_{l-i}(D) \longrightarrow \CH_l(X).$$
  Altogether, this implies that the map $\Phi_*$ is surjective.
\end{proof}

We can now gather the statements and proofs of Propositions \ref{inj}
and \ref{P:surj} into the following.

\begin{theorem} \label{main} Let $f : X \r B$ be a flat projective
  surjective morphism onto a smooth quasi-projective variety $B$ of
  dimension $d_B$.  Let $l \geq 0$ be an integer.   Assume
  that
  \begin{center}
    
  $\CH_{l-i}(X_b)= \Q$ for all $0 \leq i \leq \min (l,d_B)$ and for
  all points $b$ in $B(\Omega)$.
  \end{center}
  Then the map
  $$\Phi_* =  \bigoplus_{i=0}^{d_X-d_B} h^{d_X -d_B -i} \circ f^* \ : \
  \bigoplus_{i=0}^{d_X-d_B} \CH_{l-i}(B) \longrightarrow \CH_l(X)$$ is
  an isomorphism. Moreover the map $$\Psi_* =
  \bigoplus_{i=0}^{d_X-d_B} f_* \circ h^i \ : \ \CH_l(X)
  \longrightarrow \bigoplus_{i=0}^{d_X-d_B} \CH_{l-i}(B)$$ is also an
  isomorphism. \qed
\end{theorem}
\begin{proof} Let $B_i$ be an irreducible closed subscheme of $B$ of
  dimension $i$ with $0 \leq i \leq \min (l,d_B)$. Since
  $\CH_{l-i}(X_b)= \Q$ for all points $b \in B(\Omega)$, Lemma
  \ref{L:universal} gives $\CH_{l-i}(X_{{\eta}_{B_i}})= \Q$. Thus the
  theorem follows from a combination of Proposition \ref{inj} (and its
  proof) and Proposition \ref{P:surj}.
\end{proof}

\section{On the motive of smooth projective varieties fibred by
  quadrics}
\label{quadrics}

Let us first recall the following result.

\begin{proposition}[Corollary 2.2 in \cite{VialCK}]
  \label{effective-coro} Let $m$ and $n$ be positive integers.  Let
  $(Y,q)$ be a motive over $k$ such that $\CH_i(Y_\Omega,q_\Omega) = 0$
  for all $i < n$ and $\CH_j(Y_\Omega, {}^tq_\Omega) = 0$ for all $j <
  m$. Then there exist a smooth projective variety $Z$ over $k$ of
  dimension $d_X - m - n$ and an idempotent $r \in
  \End(\h(Z))$ such that $(Y,q)$ is isomorphic to $(Z, r, n)$.\qed
\end{proposition}

The main result of this section is the following theorem.

\begin{theorem} \label{T:gen-intro-text} Let $f : X \r B$ be a flat
  morphism of smooth projective varieties over $k$. Assume that there
  exists a positive integer $n$ such that $\CH_l(X_b)=\Q$ for all $0 \leq l <
  n$ and for all points $b \in B(\Omega)$. Then there exists a smooth
  projective variety $Z$ of dimension $d_X - 2n$ and an idempotent $r
  \in \End(\h(Z))$ such that the motive of $X$ admits a direct sum
  decomposition $$\h(X) \cong  \bigoplus_{i=0}^{d_X-d_B}
  \h(B)(i) \oplus (Z,r,n).$$
\end{theorem}
\begin{proof} With the notations of Theorem \ref{T:mainrevisited} and
  its proof the endomorphism $\Psi \circ \Phi \in \mathrm{End}\,
  \big(\bigoplus_{i=0}^{d_X-d_B} \h(B)(i) \big)$ admits an inverse
  denoted $\Xi$.  Proposition \ref{P:selfdualidempotent} then states
  that $p := \Phi \circ \Xi \circ \Psi \in
 \End(\h(X))$ is a self-dual idempotent such that $(X,p) \cong
 \bigoplus_{i=0}^{d_X-d_B} \h(B)(i).$ By Theorem \ref{main},
 $(p_\Omega)_* : \CH_l(X_\Omega) \rightarrow \CH_l(X_\Omega)$ is an
 isomorphism for all $l<n$. It follows that
 $\CH_{l}(X_\Omega,p_\Omega) = \CH_l(X_\Omega)$ for all $l<n$ and thus
 that $\CH_{l}(X_\Omega,\id_\Omega - p_\Omega) = 0$ for all
 $l<n$. Because $p={}^tp$, we also have $\CH_{l}(X_\Omega,\id_\Omega -
 {}^tp_\Omega) = 0$ for all $l<n$.  Proposition \ref{effective-coro}
 then yields the existence of a smooth projective variety $Z$ of
 dimension $d_X - 2n$ such that $(X,\id-p)$ is isomorphic to a direct
 summand of $\h(Z)(n)$.
\end{proof}

Our original motivation was to establish Murre's conjectures
\cite{Murre1} for smooth projective varieties fibred by quadrics over
a surface. The importance of Murre's conjectures was demonstrated by
Jannsen who proved \cite{Jannsen} that these hold true for all smooth
projective varieties if and only if Bloch and Beilinson's conjecture
holds true.  In our covariant setting, Murre's conjectures can be
stated as follows. \medskip

(A) $X$ has a Chow--K\"unneth decomposition $\{\pi_0, \ldots,
\pi_{2d}\}$ : There exist mutually orthogonal idempotents $\pi_0,
\ldots, \pi_{2d} \in \CH_{d_X}(X \times X)$ adding to the identity
such that $(\pi_i)_*\mathrm{H}_*(X)=\mathrm{H}_i(X)$ for all $i$.

(B) $\pi_0, \ldots, \pi_{2l-1},\pi_{d+l+1}, \ldots, \pi_{2d}$ act
trivially on $\CH_l(X)$ for all $l$.

(C) $F^i\CH_l(X) := \ker(\pi_{2l}) \cap \ldots \cap
\ker(\pi_{2l+i-1})$ doesn't depend on the choice of the $\pi_j$'s.
Here the $\pi_j$'s are acting on $\CH_l(X)$.

(D) $F^1\CH_l(X) = \CH_l(X)_\hom := \ker \big(\CH_l(X) \r \mathrm{H}_{2l}(X)
\big).$ 

\begin{definition}
  \label{D:Murre} A variety $X$ that satisfies conjectures (A), (B)
  and (D) is said to have a \emph{Murre decomposition}.
\end{definition}

In the particular case when $f$ is a flat morphism whose geometric
fibres are quadrics\footnote{Actually if $f$ is flat and if its closed
  geometric fibres are quadrics, then all of its geometric fibres are
  quadrics. Conversely if the geometric fibres of $f$ are quadrics of
  dimension $d_X-d_B$, then $f$ is flat.}, Theorem
\ref{T:gen-intro-text} implies the following corollary. We write $\lfloor a
\rfloor$ for the greatest integer which is smaller than or equal to
the rational number $a$.

\begin{corollary} \label{C:mainquadric} Let $f : X \r B$ be a flat
  morphism of smooth projective varieties over $k$. Assume that
  $\CH_l(X_b)=\Q$ for all $0 \leq l < \frac{d_X-d_B}{2}$ and for all
  points $b \in B(\Omega)$. For instance, the geometric fibres of $f$
  could either be quadrics or complete intersection of dimension $4$
  and bidegree $(2,2)$. Then
  \begin{itemize}
  \item If $d_B=1$, then $X$ is Kimura finite-dimensional
    \cite{Kimura}.
  \item If $d_B \leq 2$, then $X$ has a Murre decomposition.
 \item If $d_B=3$, $d_X-d_B$ is odd and $B$ has a Murre decomposition,
   then $X$ has a Murre decomposition.
  \end{itemize}
\end{corollary}
\begin{proof} By Theorem \ref{T:gen-intro-text}, there is a variety $Z$ and an
  idempotent $r \in \End(\h(Z))$ such that the motive of $X$ admits a
  direct sum decomposition $$\h(X) \cong \bigoplus_{i=0}^{d_X-d_B}
  \h(B)(i) \oplus (Z,r,\lfloor \frac{d_X-d_B+1}{2}\rfloor),$$ where
  \[d_Z = \left\{ \begin{array}{ll} d_B-1 & \mbox{if $d_X-d_B$ is odd}; \\
      d_B & \mbox{if $d_X-d_B$ is even}.\end{array} \right.\] Thus, we
  only need to note that any direct summand of the motive of a curve
  is finite-dimensional \cite{Kimura} and that any direct summand of
  the motive of a surface has a Murre decomposition \cite[Theorem
  3.5]{VialCK}. Finally, let us mention that, when $d_B=1$, it is not
  necessary to assume $f$ to be flat to conclude that $X$ is Kimura
  finite-dimensional; see Propositions \ref{quadrics1} and \ref{22}
  below.
\end{proof}

\begin{remark}
  Examples of $3$-folds having a Murre decomposition include products
  of a curve with a surface \cite{Murre2}, $3$-folds rationally
  dominated by a product of curves \cite{Vial4} and uniruled $3$-folds
  \cite{dAMS}.
\end{remark}

\begin{remark}[The case of smooth families]
  Suppose $f : X \r B$ is a smooth morphism between smooth projective
  varieties with geometric fibres being quadric hypersurfaces. 
    Iyer \cite{Iyer2} showed that $f$ is \'etale locally trivial and
  deduced that $f$ has a relative Chow--K\"unneth decomposition. By
  using the technique of Gordon--Hanamura--Murre \cite{GHM2},
  it is then possible to prove that 
  \[ \h(X) \cong \left\{ \begin{array}{ll} \bigoplus_{l=0}^{d_X-d_B}
      \h(B)(l) & \mbox{if $d_X-d_B$ is odd}; \\
      \bigoplus_{l=0}^{d_X-d_B} \h(B)(l) \oplus
      \h(B)(\frac{d_X-d_B}{2}) & \mbox{if $d_X-d_B$ is
        even}.\end{array} \right.\]
\end{remark}

\begin{remark} Suppose $f : X \r S$ is a complex morphism from a
  smooth projective $3$-fold $X$ to a smooth projective surface $S$
  whose fibres are conics. In that case, Nagel and Saito
  \cite{Nagel-Saito} identify (up to direct summands isomorphic to
  $\mathds{1}$ or $\mathds{1}(1)$) the motive $(Z,r)$ in the proof of
  Corollary \ref{C:mainquadric} with the $\h_1$ of the Prym variety
  $P$ attached to a double-covering of the discriminant curve $C$ of
  $f$. If now $f : X \r S$ is a flat complex morphism from a smooth
  projective variety $X$ to a smooth projective surface $S$ whose
  fibres are odd-dimensional quadrics, then, because the motive of a
  curve is Kimura finite-dimensional and by the Lefschetz
  $(1,1)$-theorem, one would deduce an identification of the $\h_1$ of
  $(Z,r)$ with $\h_1(P)$ from an isomorphism of Hodge structures
  $\mathrm{H}^1(Z,r) \cong \mathrm{H}^1(P)$. Here, $P$ again is the
  Prym variety attached to a double-covering of the discriminant curve
  $C$ of $f$.  Such an identification is currently being investigated
  by J. Bouali \cite{Bouali} by generalising the methods of
  \cite{Nagel-Saito}.
\end{remark}

\begin{corollary} \label{C:ratnum} Let $f : X \r B$ be a flat dominant
  morphism between smooth projective varieties defined over a finite
  field $\mathbf{F}$ whose geometric fibres are quadrics. If $d_B \leq
  2$, then numerical and rational equivalence agree on $X$.
\end{corollary}
\begin{proof} As in the proof of Corollary \ref{C:mainquadric}, there
  is a direct sum decomposition \begin{equation} \label{E:finite}
    \h(X) \cong \bigoplus_{i=0}^{d_X-2} \h(B)(i) \oplus (Z,r, \lfloor
    \frac{d_X-d_B+1}{2} \rfloor)
  \end{equation} 
  for some smooth projective variety $Z$, which is a curve if
  $d_X-d_B$ is odd and a surface if $d_X-d_B$ is even. Now the action
  of correspondences preserves numerical equivalence so that if
  $\alpha$ denotes the isomorphism from $\h(X)$ to the right-hand side
  of \eqref{E:finite} and if $\beta$ denotes its inverse, then we have
  $\CH_l(X)_{\mathrm{num}} = \beta_*\alpha_*\CH_l(X)_{\mathrm{num}}$
  for all $l$. In particular, $\CH_l(X)_{\mathrm{num}} =
  \beta_*\big(\bigoplus_{i=0}^{d_X-2} \CH_{l-i}(B)_{\mathrm{num}}
  \oplus r_*\CH_{l-m}(Z)_{\mathrm{num}}\big)$, where $m= \lfloor
  \frac{d_X-d_B+1}{2} \rfloor$. The corollary then follows from the
  fact that for any smooth projective variety $Y$ defined over a
  finite field the groups $\CH_0(Y)_{\mathrm{num}}$,
  $\CH^1(Y)_{\mathrm{num}}$ and $\CH^0(Y)_{\mathrm{num}}$ are zero.
\end{proof}

\section{On the motive of a smooth blow-up}

Let $X$ be a smooth projective variety over a field $k$ and let $j : Y
\hookrightarrow X$ be a smooth closed subvariety of codimension $r$.
We write $\tau : \X_Y \r X$ for the blow-up of $X$ along $Y$. Manin
\cite{Manin} showed by an existence principle that the natural map,
which is denoted $\Phi$ below, $\h(X) \oplus \bigoplus_{i=1}^{r-1}
\h(Y)(i) \longrightarrow \h(\X_Y)$ is an isomorphism of Chow
motives. Here, we make explicit the inverse to this isomorphism. An
application is given by Proposition \ref{lem birat invariant}. The
results of this section will not be used in the rest of the
paper.\medskip

We have the following fibre square $$\xymatrix{D \ar[r]^{\j}
  \ar[d]_{\tau_D} & \X_Y \ar[d]^{\tau} \\ Y \ar[r]^j & X}$$ where $\j
: D \r \X_Y$ is the exceptional divisor and where $\tau_D : D \r Y$ is
a $\P^{r-1}$-bundle over $Y$. Precisely $D = \P(\mathcal{N}_{Y/X})$ is
the projective bundle over $Y$ associated to the normal bundle
$\mathcal{N}_{Y/X}$ of $Y$ inside $X$.  The tautological line bundle
on $D = \P(\mathcal{N}_{Y/X})$ is
$\mathcal{O}_{\P(\mathcal{N}_{Y/X})}(-1) = \mathcal{O}_{\X_Y}(D)|_D$.
Let
$$H^{r-1} \subset \ldots \subset H^i \subset \ldots \subset H
\subset D$$ be linear sections of $D$ corresponding to the relatively
ample line bundle $\mathcal{O}_{\P(\mathcal{N}_{Y/X})}(1)$, where
$H^i$ has codimension $i$. Thus, if $D \hookrightarrow \P^M \times Y$
is an embedding over $Y$ corresponding to
$\mathcal{O}_{\P(\mathcal{N}_{Y/X})}(1)$, then $ H^i$ denotes the
smooth intersection of $D$ with $L^i \times Y$ for some linear
subspace $L^i$ of codimension $i$ inside $\P^M$. Let us write $\iota^i
: H^i \hookrightarrow D$ for the inclusion maps.

As we will be using repeatedly Manin's identity principle, let us
mention that if $Z$ is a smooth projective variety over $k$, then the
blow-up of $X \times Z$ along $Y \times Z$ canonically identifies with
$\tau \times \id_Z : \X_Y \times Z \r X \times Z$.  We write $H^i_Z$
for $H^i \times Z$.\medskip

Let us define the morphism of motives $$\Phi := {}^t\Gamma_\tau \oplus
\bigoplus_{i=1}^{r-1} \Gamma_{\j} \circ h^{r-1-i} \circ
{}^t\Gamma_{\tau_D} \ : \ \h(X) \oplus \bigoplus_{i=1}^{r-1} \h(Y)(i)
\longrightarrow \h(\X_Y).$$ Here, $h^l$ is the correspondence
$\Gamma_{\iota^l} \circ {}^t \Gamma_{\iota^l}$; it coincides with
the $l$-fold composite of $h := \Gamma_{\iota^1} \circ {}^t
\Gamma_{\iota^1}$ with itself.

On the one hand, we have the following blow-up formula for Chow
groups; see \cite{Manin}.

\begin{proposition} \label{P:projbundleChow} The induced map $$\Phi_* =
  \tau^* \oplus \bigoplus_{i=1}^{r-1} {\j}_* h^{r-1-i} {\tau_D}^* \ :
  \ \CH_l(X) \oplus \bigoplus_{i=1}^{r-1} \CH_{l-i}(Y) \longrightarrow
  \CH_l(\X_Y)$$
  is an isomorphism. \qed
\end{proposition}

On the other hand, we define
$$\Psi := \Gamma_\tau \oplus \bigoplus_{i=1}^{r-1} (-1) \cdot 
\Gamma_{\tau_D} \circ h^{i-1} \circ {}^t\Gamma_{\j} \ : \ \h(\X_Y)
\longrightarrow \h(X) \oplus \bigoplus_{i=1}^{r-1} \h(Y)(i).$$ Let
$(\Psi \circ \Phi)_{i,j}$ be the $(i,j)^\mathrm{th}$ component of
$\Psi \circ \Phi$, where $\h(X)$ is by definition the $0^\mathrm{th}$
coordinate of $\h(X) \oplus \bigoplus_{i=1}^{r-1} \h(Y)(i)$. Thus, if
$i, j \neq 0$, then $(\Psi \circ \Phi)_{i,j}$ is a morphism $\h(Y)(j)
\r \h(Y)(i)$; if $i \neq 0$, then $(\Psi \circ \Phi)_{i,0}$ is a
morphism $\h(X) \r \h(Y)(i)$; if $j \neq 0$, then $(\Psi \circ
\Phi)_{0,j}$ is a morphism $\h(Y)(j) \r \h(X)$; and $(\Psi \circ
\Phi)_{0,0}$ is a morphism $\h(X) \r \h(X)$.

The following lemma shows that $\Psi \circ \Phi$ is a lower
triangular matrix with invertible diagonal elements.

\begin{lemma} \label{P:triang}
  We have
$$
(\Psi \circ \Phi)_{i,j} = \left\{ \begin{array}{cl}
    0 &\mbox{ if $i <j$}  \\
    0 &\mbox{ if $ij=0$ unless $i=j=0$} \\
    \Delta_X &\mbox{ if $i=j=0$}  \\
    \Delta_Y &\mbox{ if $i=j>0$}
       \end{array} \right.
$$
\end{lemma}
\begin{proof}
The proposition consists of the following relations:
  \begin{enumerate}
\item $\Gamma_\tau \circ {}^t\Gamma_\tau = \Delta_X$.
\item $\Gamma_{\tau_D} \circ h^{i-1} \circ {}^t\Gamma_{\j} \circ
  \Gamma_{\j} \circ h^{r-1-j} \circ {}^t\Gamma_{\tau_D} = 0$ for all
  $1 \leq i < j \leq r-1$.
\item $\Gamma_{\tau_D} \circ h^{i-1} \circ {}^t\Gamma_{\j} \circ
  \Gamma_{\j} \circ h^{r-1-i} \circ {}^t\Gamma_{\tau_D} = - \Delta_Y$
  for all $1 \leq i \leq r-1$.
\item $\Gamma_\tau \circ \Gamma_{\j} \circ h^{r-1-i} \circ {}^t
  \Gamma_{\tau_D} = 0$ for all $1 \leq i \leq r-1$.
\end{enumerate}
Let us establish them. The morphism $\tau$ is a birational morphism so
that the identity (1) follows from the projection formula as in the
proof of Lemma \ref{dominant}.
The proof of (4) is a combination of the fact that $\tau \circ \j = j
\circ \tau_D$ and Lemma \ref{orth}.  As for (2) and (3), we claim that
$$^t\Gamma_{\j} \circ \Gamma_{\j} = - \Gamma_{\iota} \circ
{}^t\Gamma_{\iota} = -h \in \CH_{d_X-1}(D \times D).$$ Indeed, the
action of $h$ on $\CH_*(D)$ is given by intersecting with the class of
$H$.  Also, by \cite[Prop.  2.6]{Fulton}, the map $\j^*\j_* :
\CH_{*}(D) \r \CH_{*-1}(D)$ is given by intersecting with the class of
$D|_D$ which is precisely $-h$. The same arguments for the smooth
blow-up of $X \times Z$ along $Y \times Z$ (whose exceptional divisor
is $D \times Z \hookrightarrow \X_Y \times Z$), together with Manin's
identity principle, yield the claim.

In view of the above claim, (2) follows from Lemma \ref{orth} and (3)
follows from Lemma \ref{dominant}.
\end{proof}

Thus the endomorphism $$N := \left( \begin{array}{ccccc}
    \Delta_X & 0 & 0 & \cdots & 0 \\
    0 & \Delta_Y & 0 & \cdots & 0 \\
    0 & 0 & \ddots & \ddots & \vdots \\
    \vdots & \vdots & \ddots & \ddots & 0 \\ 0 & 0 & \cdots & 0 &
    \Delta_Y
  \end{array} \right) - \ \Psi \circ \Phi $$ is a nilpotent
endomorphism of $\h(X) \oplus \bigoplus_{i=1}^{r-1} \h(Y)(i)$ of index
$\leq r-1$, \emph{i.e.} $N^{r-1} = 0$. The morphism $$\Theta := (\id + N +
\ldots + N^{r-2}) \circ \Psi$$ thus gives a left-inverse to
$\Phi$. \medskip

The main result of this section is then the following theorem.

\begin{theorem} \label{T:blowup}
  The morphism $\Theta$ is the inverse of $\Phi$.
\end{theorem}

\begin{proof}
  Let $p := \Phi \circ \Theta \in \End(\h(\X_Y))$. Because $\Theta$ is
  a left-inverse to $\Phi$, we see that $p$ is an idempotent. The
  motive of $\X_Y$ thus splits as $$\h(\X_Y) = (\X_Y,p) \oplus
  (\X_Y,\id-p).$$ As a consequence of Proposition
  \ref{P:projbundleChow} and Lemma \ref{P:triang}, we obtain that
$$\CH_* (\X_Y,p)= \CH_*(\X_Y).$$

Actually, since $\widetilde{(X \times Z)}_{Y \times Z}$ canonically
identifies with $\X_Y \times Z$, we get, thanks to \cite[Prop.
16.1.1]{Fulton}, that $\Phi \times \Delta_Z$ induces an isomorphism of
Chow groups as in Proposition \ref{P:projbundleChow} and that $\Theta
\times \Delta_Z$ is a left-inverse to $\Phi \times \Delta_Z$. Thus, we
also have that $$\CH_* (\X_Y \times Z,p \times \Delta_Z)= \CH_*(\X_Y
\times Z).$$ Therefore $$\CH_* (\X_Y \times Z,\Delta_{\X_Y} \times
\Delta_Z - p \times \Delta_Z)= 0.$$ By Manin's identity principle it
follows that
$$p = \Delta_{\X_Y}.$$ In other words, $\Theta$ is not only a
left-inverse to $\Phi$, it is also the inverse of $\Phi$.
\end{proof}

Let us now use Theorem \ref{T:blowup} to study the birational
invariance of some groups of algebraic cycles attached to smooth
projective varieties.  For a smooth projective variety $X$ over $k$,
we write $\mathrm{Griff}_l(X)$ for its Griffiths group
$\CH_l(X)_\hom/\CH_l(X)_{\alg}$. We also write, when $k \subseteq \C$,
$$T^l(X) := \ker\big( AJ^l : \CH^l(X)_{\hom} \rightarrow J^l(X)
\big)$$ for the kernel of Griffiths' Abel--Jacobi map to the
intermediate Jacobian $J^l(X)$ which is a quotient of
$\mathrm{H}^{2l-1}(X,\mathbb{C})$.

If $\pi : \widetilde{X} \rightarrow X$ is a birational map, the
projection formula implies that $\Gamma_\pi \circ {}^t\Gamma_\pi =
\Delta_X$; see Lemma \ref{dominant}. Thus $\pi_*\pi^*$ acts as the
identity on $\CH_l(X)$, $\mathrm{Griff}_l(X)$ and on $T^l(X)$. The
following proposition shows that in some cases $\pi_*$ and $\pi^*$ are
actually inverse to each other.

\begin{proposition}\label{lem birat invariant}
  Let $\pi : \widetilde{X} \rightarrow X$ be a birational map between
  smooth projective varieties.  Then $\pi^* \pi_*$ acts as the
  identity on $\CH_0(\widetilde{X})$,
  $\mathrm{Griff}_1(\widetilde{X})$,
  $\mathrm{Griff}^2(\widetilde{X})$, $T^2(\widetilde{X})$,
  $\CH^1(\widetilde{X})_{\hom}$ and $\CH^0(\widetilde{X})$.
\end{proposition}
\begin{proof} By resolution of singularities, there are morphisms $f :
  Y \r \X$ and $g : Y \r X$, which are composite of smooth blow-ups,
  such that $g = \pi \circ f$. The groups considered in the
  proposition behave functorially with respect to the action of
  correspondences. Therefore it is enough to prove the proposition
  when $\pi$ is a smooth blow-up $\X_Y \r X$ as above. First note that
  $\Psi_i := \h(\X_Y) \r \h(Y)(i)$ acts as zero on the groups
  considered in the proposition when $1 \leq i \leq r-1$, for
  dimension reasons.  Having in mind that the first column and the
  first row of $N$ are zero, and expanding $\Phi \circ \Theta$, we see
  that $\Phi \circ \Theta$ acts like $\pi^*\pi_*$ on the groups of the
  proposition. By Theorem \ref{T:blowup}, the correspondence $\Phi
  \circ \Theta$ acts as the identity on $\CH_l(\X)$. Thus $\pi^*\pi_*$
  acts as the identity on the groups of the proposition.
\end{proof}

\begin{remark}
  As a consequence of Theorem \ref{T:blowup}, we obtain an explicit
  Chow--K\"unneth decomposition for a smooth blow-up $\widetilde{X}_Y$
  in terms of Chow--K\"unneth decompositions of $X$ and
  $Y$.
  Precisely, 
  assume that $X$ and $Y$ are endowed with Chow--K\"unneth
  decompositions $\{\pi^i_X, 0 \leq i \leq 2\dim X\}$ and $\{\pi^i_Y,
  0 \leq i \leq 2\dim Y\}$, respectively.
  Let us define
$$\pi^i_{\X_Y} := \Phi \circ \big( \pi_X^i \oplus  \bigoplus_{j=1}^{r-1}
\pi_Y^{i-2j} \big)
\circ \Theta \in \End\big( \h(\X_Y) \big).$$ Then $\{\pi^i_{\X_Y}, 0
\leq i \leq 2\dim \X_Y\}$ is a Chow--K\"unneth decomposition of
$\X_Y$.
\end{remark}

\section{On the Chow groups of varieties fibred by varieties with
  small Chow groups}
\label{general}

In this section, we consider a projective surjective morphism $f : X
\r B$ onto a quasi-projective variety $B$. We are interested in
obtaining information on the niveau of the Chow groups of $X$, under
the assumption that the Chow groups of the fibres of $f$ over generic
points of closed subvarieties of $B$ are either trivial ($=\Q$) or
finitely generated. Contrary to Section \ref{hypsection}, we do not
assume that $f$ is flat. Let $l$ be a non-negative integer and let
$d_X$ and $d_B$ be the dimensions of $X$ and $B$, respectively. In \S
6.1, we assume that $\CH_{l-i}({X}_{\eta_{D_i}}) = \Q$ for all $0< i
\leq d_B$ and all irreducible subvarieties $D_i \subset B$ of
dimension $i$, and we deduce in Lemma \ref{localisation-seq} by a
localisation sequence argument that $\CH_l(X)$ is spanned by
$\CH_l(X_b)$ for all closed points $b$ in $B$ and by $\CH_l(H)$, where
$H \hookrightarrow X$ is a linear section of $X$ of dimension $d_X +
l$. We then move on to study the subspace of $\CH_l(X)$ spanned by
$\CH_l(X_b)$ for all closed points $b$ in $B$ in the case when
$\CH_l(X_b) =\Q$ in \S 6.2, and in the case when $\CH_l(X_b)$ finitely
generated in \S 6.3. The results of \S\S 6.1 \& 6.2 are combined into
Proposition \ref{linear-theorem}, while the results of \S\S 6.1 \& 6.3
are combined into Propositions \ref{1} and \ref{genth} in \S
6.4. Finally, in \S 6.5, we use Lemma \ref{lemma} to give statements
that only involve the Chow groups of closed fibres when $f$ is defined
over the complex numbers; see Theorems \ref{linear-theorem-complex},
\ref{1-complex} and \ref{genth-complex}.

\subsection{Some general statements}
\begin{lemma} \label{localisation-seq} Let $f : X \r B$ be a
  projective surjective morphism onto a quasi-projective variety $B$
  and let $H \hookrightarrow X$ be a linear section of dimension $\geq
  l+d_B$.  Assume that
  $\CH_{l-i}({X}_{\eta_{D_i}}) = \Q$ for all $0< i \leq d_B$ and all
  irreducible subvarieties $D_i \subset B$ of dimension $i$. Then the
  natural map $\bigoplus_{b \in B} \CH_l(X_b) \oplus \CH_l(H) \r
  \CH_l(X)$ is surjective.
\end{lemma}
\begin{proof}
  We prove the proposition by induction on $d_B$. If $d_B=0$ then the
  statement is obvious. Let us thus consider a morphism $f : X \r B$
  and a linear section $\iota : H \hookrightarrow X$ as in the
  statement of the proposition with $d_B>0$. By Lemma
  \ref{L:dominant}, $f$ restricted to $H$ is surjective. We have the
  localisation exact sequence
  $$\bigoplus_{D \in B^{1}} \CH_l(X_D) \r \CH_l(X) \r
  \CH_{l-d_B}(X_{\eta_B}) \r 0.$$ Here, $B^{1}$ denotes the set of
  codimension-one closed irreducible subschemes of $B$.  For any
  irreducible codimension-one subvariety $D \subset B$, the
  restriction of $\iota$ to $D \r B$ defines a linear section $\iota_D
  : H_D \hookrightarrow X_D$ of dimension $\geq l +d_B - 1$ of $X_D$.
  The restriction of $f : X \r B$ to $D \r B$ defines a surjective
  morphism $X_D \r D$ which together with the linear section $\iota_D$
  satisfies the assumptions of the proposition.  Therefore, by the
  induction hypothesis applied to $X_D \r D$, the map
  $$\bigoplus_{d \in D} \CH_l(X_d) \oplus \CH_l(H_D) \r \CH_l(X_D)$$ is
  surjective. This yields an exact sequence
  $$\bigoplus_{b \in B} \CH_l(X_b) \oplus \bigoplus_{D \in B^{1}}
  \CH_l(H_D) \r \CH_l(X) \r \CH_{l-d_B}(X_{\eta_B}) \r 0.$$ Since each of
  the proper inclusion maps $H_D \r X$ factors through $\iota : H \r
  X$, we see that the map $\bigoplus_{D \in B^{1}} \CH_l(H_D)
  \stackrel{\oplus (\iota_D)_*}{\longrightarrow} \CH_l(X)$ factors
  through $\iota_* : \CH_l(H) \r \CH_l(X)$. In order to conclude, it is
  enough to prove that the composite map $$\CH_l(H) \r \CH_l(X) \r
  \CH_{l-d_B}(X_{\eta_B})$$ is surjective.  If $l < d_B$, then this is
  obvious. Let us then assume that $l \geq d_B$.

  Let $Y$ be an irreducible subvariety of $H$ of dimension $l$ such
  that the composite $Y \hookrightarrow X \r B$ is dominant.  Because
  $\CH_{l-d_B}(X_{\eta_B}) = \Q$ it is enough to see that the class of
  $Y$ in $\CH_l(X)$ maps to a non-zero element in
  $\CH_{l-d_B}(X_{\eta_B})$. But, as in the proof of Proposition
  \ref{P:surj}, $[Y]$ maps to $[Y_{\eta_B}] \neq 0 \in
  \CH_{l-d_B}(X_{\eta_B})$.
\end{proof}

Here is an improvement of Lemma \ref{localisation-seq}:

\begin{lemma} \label{localisation-seq2} Let $f : X \r B$ be a
  projective surjective morphism onto a quasi-projective variety $B$
  and let $H \hookrightarrow X$ be a linear section of dimension $\geq
  l+d_B$.  Assume that: \medskip

  $\bullet$ $\CH_{l-i}({X}_{\eta_{D_i}}) = \Q$ for all $i$ such that $0
  <i < d_B$ and all irreducible subvarieties $D_i \subset B$ of
  dimension $i$.

  $\bullet$ $\CH_{l-d_B}({X}_{\eta_{B}})$ is finitely generated. 
  \medskip

  \noindent Then there exist finitely many closed subschemes
  $\mathcal{Z}_j$ of $X$ of dimension $l$ such that the natural map
  $\bigoplus_j \CH_l(\mathcal{Z}_j) \oplus \bigoplus_{b \in B}
  \CH_l(X_b) \oplus \CH_l(H) \r \CH_l(X)$ is surjective.
\end{lemma}

\begin{proof}
  If $d_B=0$ then the statement is obvious. Let us thus consider a
  morphism $f : X \r B$ and a linear section $\iota : H
  \hookrightarrow X$ as in the statement of the proposition with
  $d_B>0$. The morphism $f$ restricted to $H$ is surjective; see Lemma
  \ref{L:dominant}. As in the proof of Lemma
  \ref{localisation-seq} we have the localisation exact sequence
  $$\bigoplus_{D \in B^{1}} \CH_l(X_D) \r \CH_l(X) \r
  \CH_{l-d_B}(X_{\eta_B}) \r 0.$$ Each of the morphisms $X_D \r D$
  satisfies the assumptions of Lemma \ref{localisation-seq} and
  by the same arguments as in the proof of Lemma
  \ref{localisation-seq} we get that the image of the map
  $\bigoplus_{b \in B} \CH_l(X_b) \oplus \CH_l(H) \r \CH_l(X)$ contains
  the image of the map $\bigoplus_{D \in B^{1}} \CH_l(X_D) \r \CH_l(X)$.
  Let now $Z_j$ be finitely many closed subschemes of $X_{\eta_B}$
  whose classes $[Z_j] \in \CH_{l-d_B}(X_{\eta_B})$ generate $
  \CH_{l-d_B}(X_{\eta_B})$. By surjectivity of the map $\CH_l(X) \r
  \CH_{l-d_B}(X_{\eta_B})$ there are cycles $\alpha_j \in \CH_l(X)$ that
  map to $[Z_j]$. If $\mathcal{Z}_j$ is the support in $X$ of any
  representative of $\alpha_j$, we then have a surjective map
  $\bigoplus_j \CH_l(\mathcal{Z}_j) \r \CH_l(X) \r
  \CH_{l-d_B}(X_{\eta_B})$. It is then clear that the map $\bigoplus_j
  \CH_l(\mathcal{Z}_j) \oplus \bigoplus_{b \in B} \CH_l(X_b) \oplus
  \CH_l(H) \r \CH_l(X)$ is surjective.
\end{proof}

\subsection{Varieties fibred by varieties with Chow groups generated
  by a linear section}

\begin{lemma} \label{linear-niveau}
  Let $f : X \r B$ be a projective surjective morphism onto a
  quasi-projective variety $B$.  Assume that $\CH_l({X}_b) = \Q$ for
  all closed points $b \in B$.  Then, if $H \hookrightarrow X$ is a
  linear section of dimension $\geq l+d_B$, we have
  $$\im \big(\bigoplus_{b \in B} \CH_l(X_b) \r \CH_l(X) \big) \
  \subseteq \ \im \big( \CH_l(H) \r \CH_l(X) \big). $$
\end{lemma}
\begin{proof}
  Let $b$ be a closed point of $B$ and fix $H \hookrightarrow X$ a
  linear section of dimension $\geq l+d_B$. The morphism $f$
  restricted to $H$ is surjective; see Lemma \ref{L:dominant}. Let
  $Z_l$ be an irreducible closed subscheme of $X$ of dimension $l$
  which is supported on $X_b$. Since $f|_H : H \r B$ is a dominant
  projective morphism, its fibre $H_b$ over $b$ is non-empty and has
  dimension $\geq l $. By assumption $\CH_l({X}_b)=\Q$, so that a
  rational multiple of $[Z_l]$ is rationally equivalent to an
  irreducible closed subscheme of $H_b$ of dimension $l$. Therefore
  $[Z_l] \in \CH_l(X_b)$ belongs to the image of the natural map
  $\CH_l(H_b) \r \CH_l(X_b)$. Thus the image of $\CH_l(X_b) \r
  \CH_l(X) $ is contained in the image of $\CH_l(H) \r \CH_l(X)$.
\end{proof}

\begin{remark}
  It is interesting to decide whether or not it is possible to
  parametrise such $l$-cycles by a variety of dimension $d_B$; see
  Proposition \ref{hilbert}.
\end{remark}

\begin{proposition} \label{linear-theorem} Let $f : X \r B$ be a
  projective surjective morphism onto a quasi-projective variety $B$.
  Assume that $\CH_{l-i}({X}_{\eta_{D_i}}) = \Q$ for all $0 \leq i
  \leq d_B$ and all irreducible subvarieties $D_i \subset B$ of
  dimension $i$.  Then, if $H \hookrightarrow X$ is a linear section
  of dimension $\geq l+d_B$, the pushforward map $ \CH_l(H) \r
  \CH_l(X)$ is surjective.
  In particular, $\CH_l(X)$ has niveau $\leq d_B$.
\end{proposition}
\begin{proof} This is a combination of Lemma
  \ref{localisation-seq} and Lemma \ref{linear-niveau}.
\end{proof}

\subsection{An argument involving relative Hilbert schemes}

Let $f : X \r B$ be a generically smooth, projective morphism defined
over the field of complex numbers $\C$ onto a smooth
quasi-projective variety $B$.
Let $B^\circ \subseteq B$ be the smooth locus of $f$ and let $f^\circ
: X^\circ \r B^\circ$ be the pullback of $f : X \r B$ along the open
inclusion $B^\circ \hookrightarrow B$ so that we have a cartesian
square
\begin{center}
  $\xymatrix{X^\circ \ar[d]_{f^\circ} \ar@{^{(}->}[r] & X \ar[d]^f \\
    B^\circ \ar@{^{(}->}[r] & B. }$
\end{center}
We assume that there is a non-negative integer $l$ such that for all
closed points $b \in B^\circ(\C)$ the cycle class map
$\CH_l(X_b) \r \mathrm{H}_{2l}(X_b)$ is an isomorphism.\medskip

Let $\pi_d : {\mathrm{Hilb}}^{ d}_l(X/B) \r B$ be the relative Hilbert
scheme whose fibres over the points $b$ in $B$ parametrise the closed
subschemes of $X_b$ of dimension $l$ and degree $d$, and let $p_d :
\mathcal{C}_l^{d} \r {\mathrm{Hilb}}^{d}_l(X/B)$ be the universal
family over ${\mathrm{Hilb}}^{d}_l(X/B)$; see \cite[Theorem
1.4]{Kollar}.  We have the following commutative diagram, where all
the morphisms involved are proper:
  \begin{center} $ \xymatrix{
      \mathcal{C}_l^{d} \ar[d]_{p_d} \ar[rr]^{q_d} & & X. \ar[lldd] \\
      {\mathrm{Hilb}}^{d}_l(X/B) \ar[d]_{\pi_d} & \\ B}$
  \end{center} We then consider the disjoint unions
  ${\mathrm{Hilb}}_l(X/B) := \coprod_{d \geq 0}
  {\mathrm{Hilb}}^{d}_l(X/B)$ and $\mathcal{C}_l := \coprod_{d \geq 0}
  \mathcal{C}^d_l$, and denote $\pi : {\mathrm{Hilb}}_l(X/B) \r B$, $p
  : \mathcal{C}_l \r \mathrm{Hilb}_l(X/B)$ and $q : \mathcal{C}_l \r
  X$ the corresponding maps.
  
  \noindent Let us then denote \begin{center} $\mathrm{Irr}_l(X/B) :=
    \{\mathcal{H} : \mbox{$\mathcal{H}$ is an irreducible component of
      ${\mathrm{Hilb}}^{d}_l(X/B)$ for some $d$} \}$.
 \end{center}
 For a subset $\mathcal{E} \subset \mathrm{Irr}_l(X/B) $, we define
 the following closed subscheme of $B^\circ$: $$Z_{\mathcal{E}} :=
 B^\circ \cap \bigcap_{\mathcal{H} \in \mathcal{E}}
 \pi(\mathcal{H}).$$ We say that a finite subset $\mathcal{E}$ of
 $\mathrm{Irr}_l(X/B) $ is \emph{spanning} at a point $t \in
 Z_\mathcal{E}(\C)$ if $\mathrm{H}_{2l}(X_t)$ is spanned by the set
 $\{ \, cl(q_*[p^{-1}(u)]) : u\in \mathcal{H}, \mathcal{H} \in
 \mathcal{E}, \pi(u)=t \, \}.$ Note that, given $\mathcal{H} \in
 \mathrm{Irr}_l(X/B)$ and $u,u' \in \mathcal{H}$ such that
 $\pi(u)=\pi(u')=t$, $cl(q_*[p^{-1}(u)]) = cl(q_*[p^{-1}(u')]) \in
 \mathrm{H}_{2l}(X_t)$ if $u$ and $u'$ belong to the same connected
 component in $\pi^{-1}(t)$. \medskip

 \noindent \textbf{Claim.} Let $\mathcal{E}$ be a finite subset of
 $\mathrm{Irr}_l(X/B) $ that is spanning at a closed point $t \in
 Z_\mathcal{E}(\C)$. Then, for all points $s \in Z_\mathcal{E}(\C)$
 belonging to an irreducible component of $Z_\mathcal{E}$ that
 contains $t$, $\mathrm{H}_{2l}(X_s)$ is spanned by the set $\{ \,
 cl(q_*[p^{-1}(v)]) : v\in \mathcal{H}, \mathcal{H} \in \mathcal{E},
 \pi(v)=s \, \}.$ \medskip
 
 \noindent Indeed, consider any finite subset $\mathcal{E}$ of
 $\mathrm{Irr}_l(X/B) $. The local system of $\Q$-vector spaces
 $R_{2l}(f_0)_*\Q$ on $B^\circ$ restricts to a local system
 $(R_{2l}(f_0)_*\Q)|_{Z_\mathcal{E}}$ on $Z_\mathcal{E}$. If $t$ is a
 complex point on $Z_\mathcal{E}$, let $r_t$ be the rank of the
 subspace of $\mathrm{H}_{2l}(X_t)$ spanned by $\{ \,
 cl(q_*[p^{-1}(u)]) : u\in \mathcal{H}, \mathcal{H} \in \mathcal{E},
 \pi(u)=t \, \}.$ If we see this latter set as a set of sections at
 $t$ of the local system $R_{2l}(f_0)_*\Q$, then these sections extend
 locally to constant sections of the local system
 $(R_{2l}(f_0)_*\Q)|_{Z_\mathcal{E}}$ on $Z_\mathcal{E}$. This shows
 that the rank $r_t$ is locally constant. If $\mathcal{E}$ is spanning
 at the point $t \in Z_\mathcal{E}(\C)$, then $r_t$ is maximal, equal
 to $\dim_\Q \mathrm{H}_{2l}(X_t)$. The subset of $Z_\mathcal{E}(\C)$
 consisting of points $s$ in $Z_\mathcal{E}(\C)$ for which $r_s =
 \dim_\Q \mathrm{H}_{2l}(X_s)$ is therefore both open and closed in
 $Z_\mathcal{E}$. It contains then the irreducible components of
 $Z_\mathcal{E}$ that contain $t$.

 \begin{lemma} \label{locally-closed} There exists a finite subset
   $\mathcal{E}$ of $\mathrm{Irr}_l(X/B)$ such that $B^\circ =
   Z_\mathcal{E}$ and such that $\mathcal{E}$ is spanning at every
   point $t \in B^\circ(\C)$.
\end{lemma}
\begin{proof} By working component-wise, we may assume that $B$ is
  irreducible. By assumption on $f: X \r B$, $\mathrm{H}_{2l}(X_t)$ is
  spanned by algebraic cycles on $X_t$ for all points $t \in
  B^\circ(\C)$. Thus, for all points $t \in B^\circ(\C)$, there is a
  finite subset $\mathcal{E}_t$ of $\mathrm{Irr}_l(X/B) $ that is
  spanning at $t$. For each point $t$, choose an irreducible component
  $Y_{\mathcal{E}_t}$ of $Z_{\mathcal{E}_t}$ that contains
  $t$. According to the claim above, $\mathcal{E}_t$ is spanning at
  every point $s \in Y_{\mathcal{E}_t}(\C)$. Now, we have $B^\circ(\C)
  = \coprod_{t \in B^\circ(\C)} Y_{\mathcal{E}_t}(\C)$. Since there
  are only countably many finite subsets of $\mathrm{Irr}_l(X/B)$ and
  since $Z_\mathcal{E}$ has only finitely many irreducible components,
  we see that the latter union is in fact a countable union. This
  yields that $B^\circ = Y_\mathcal{E}$ for some finite subset
  $\mathcal{E}$ of $\mathrm{Irr}_l(X/B)$ that is spanning at every
  point in $Y_\mathcal{E}(\C)$. We then conclude that $B^\circ =
  Z_\mathcal{E}$ and that $\mathcal{E}$ is spanning at every point $ t
  \in Y_\mathcal{E}(\C) = B^\circ(\C)$.
 \end{proof}

\begin{proposition} \label{hilbert} Let $f : X \r B$ be a generically
  smooth and projective morphism defined over $\C$ onto a smooth
  quasi-projective variety $B$.  Let $B^\circ \subseteq B$ be the
  smooth locus of $f$.  Assume that there is an integer $l \leq d_X -
  d_B$ such that for all closed points $b \in B^\circ(\C)$ the cycle
  class map $\CH_l(X_b) \r \mathrm{H}_{2l}(X_b)$ is an
  isomorphism. Then $\im \big(\bigoplus_{b \in B^\circ} \CH_l(X_b) \r
  \CH_l(X) \big)$ is supported on a closed subvariety of $X$ of
  dimension $d_B+l$.

  If moreover $X$ is smooth, then there exist a smooth
  quasi-projective variety $\widetilde{B}$ of dimension $d_B$ and a
  correspondence $\Gamma \in \CH_{d_B+l}(\widetilde{B} \times X)$ such
  that $\Gamma_* : \CH_0(\widetilde{B}) \r \CH_l(X)$ is well-defined
  and
  $$\im \big(\bigoplus_{b \in B^\circ} \CH_l(X_b) \r \CH_l(X) \big) \
  \subseteq \ \im \big(\Gamma_* : \CH_0(\widetilde{B}) \r \CH_l(X)
  \big). $$
\end{proposition}
\begin{proof} By Lemma \ref{locally-closed}, there exists a finite set
  $\mathcal{E}$ of irreducible components of ${\mathrm{Hilb}}_l(X/B)$
  such that $B^\circ = Z_\mathcal{E}$ and such that for all points $t
  \in B^\circ(\C)$ the set $\{ \, cl(q_*[p^{-1}(u)]) : u\in
  \mathcal{H}, \mathcal{H} \in \mathcal{E}, \pi(u)=t \, \}$ spans
  $\mathrm{H}_{2l}(X_t)$. Denote $\mathcal{H}_i$ the irreducible
  components of ${\mathrm{Hilb}}_l(X/B)$ that belong to $\mathcal{E}$
  and let $\widetilde{\mathcal{H}}_i \r \mathcal{H}_i$ be resolutions
  thereof.  For all $i$, pick a smooth linear section $\widetilde{B}_i
  \r \widetilde{\mathcal{H}}_i$ of dimension $d_B$.  Lemma
  \ref{L:dominant} shows that $r_i : \widetilde{B}_i \r
  \widetilde{\mathcal{H}}_i \r \mathcal{H}_i \r B$ is surjective and a
  refinement of its proof shows that, for all points $b \in B(\C)$,
  $r_i^{-1}(b)$ contains a point in every connected component of
  $\widetilde{\mathcal{H}}_{i,b}$.  Consider then $p_i :
  (\mathcal{C}_l^{})|_{\widetilde{B}_i} \r \widetilde{B}_i$ the
  pullback of the universal family $p: \mathcal{C}_l^{} \r
  {\mathrm{Hilb}}^{}_l(X/B)$ along $\widetilde{B}_i \hookrightarrow
  \widetilde{\mathcal{H}}_i \r {\mathcal{H}}_i \hookrightarrow
  {\mathrm{Hilb}}^{}_l(X/B)$.  For each $i$, we have the following
  picture
  \begin{center}
    $\xymatrix{ (\mathcal{C}_l^{})|_{\widetilde{B}_i} \ar[r]^{\
        \ \  q_i} \ar[d]_{p_i} & X \\ \widetilde{B}_i}$
  \end{center}
  and we have $$\im \big(\bigoplus_{b \in B^\circ} \CH_l(X_b) \r
  \CH_l(X) \big) \ \subseteq \ \sum_i \im \big((q_i)_* :
  \CH_l((\mathcal{C}_l^{})|_{\widetilde{B}_i}) \r \CH_l(X) \big)$$ so
  that the group on the left-hand side is supported on the union of
  the scheme-theoretic images of the morphisms $q_i$.

  If $X$ is smooth, we define $\Gamma_i \in
  \CH_{d_B+l}(\widetilde{B}_i \times X)$ to be the class of the image
  of $(\mathcal{C}_l^{})|_{\widetilde{B}_i}$ inside $\widetilde{B}_i
  \times X$. Because $q : \mathcal{C}_l^{d} \r X$ is proper for all $d
  \geq 0$, $\Gamma_i$ has a representative which is proper over
  $X$. It is therefore possible \cite[Remark 16.1]{Fulton} to define
  maps $(\Gamma_i)_* : \CH_0(\widetilde{B}_i) \r \CH_l(X)$ for all
  $i$. In fact, we have $(\Gamma_i)_* = (q_i)_*p_i^*$.  Finally, we
  define $\widetilde{B}$ to be the disjoint union of the
  $\widetilde{B}_i$'s and $\Gamma \in \CH_{d_B+l}(\widetilde{B} \times
  X)$ to be the class of the disjoint union of the correspondences
  $\Gamma_i$.
\end{proof}

\subsection{Complex varieties fibred by varieties with small Chow
  groups} From now on, the base field $k$ is assumed to be the field
of complex numbers $\C$.

\begin{proposition}\label{1}  Let  $f : X \r C$ be a generically smooth,
  projective morphism defined over $\C$ to a smooth curve. Assume that
  \medskip

  $\bullet$ $\CH_l({X}_c)$ is finitely generated for all closed points
  $c \in C$,

  $\bullet$ $\CH_l({X}_c) \r \mathrm{H}_{2l}(X_c)$ is an isomorphism for a
  general closed point $c \in C$,

  $\bullet$ $\CH_{l-1}({X}_\eta)$ is finitely generated, where $\eta$ is
  the generic point of $C$. \medskip

  \noindent Then $\CH_l(X)$ has niveau $\leq 1$.
\end{proposition}

\begin{proof} We have the localisation exact sequence $$\bigoplus_{c
    \in C} \CH_l(X_c) \longrightarrow \CH_l(X) \longrightarrow
  \CH_{l-1}(X_{\eta}) \longrightarrow 0.$$ Let $Z_1, \ldots, Z_n$ be
  irreducible closed subschemes of $X_\eta$ of dimension $l-1$ that
  span $\CH_{l-1}(X_\eta)$ and let $\mathcal{Z}_1, \ldots,
  \mathcal{Z}_n$ be closed subschemes of $X$ of dimension $l$ that
  restrict to $Z_1, \ldots, Z_n$ in $X_\eta$. Then by flat pullback
  the class of $\mathcal{Z}_j$ in $\CH_l(X)$ maps to the class of
  $Z_j$ in $\CH_{l-1}(X_\eta)$ so that the composite map
  $\bigoplus_{j=1}^n \CH_l(\mathcal{Z}_j) \r \CH_l(X) \r
  \CH_{l-1}(X_\eta)$ is surjective.

  Let $U \subseteq C$ be a Zariski-open subset of $C$ such that for
  all closed points $c \in U$ the cycle class map $\CH_l({X}_c) \r
  \mathrm{H}_{2l}(X_c)$ is an isomorphism. Up to shrinking $U$, we may
  assume that $f|_U : X|_U \r U$ is smooth. We may then apply
  Proposition \ref{hilbert} to get a closed subscheme $\iota : D
  \hookrightarrow X$ of dimension $l+1$ such that $\iota_*\CH_l(D)
  \supseteq \im \big( \bigoplus_{c \in U} \CH_l(X_c) \r \CH_l(X)
  \big)$.

  As such, we have a surjective map $$\bigoplus_{j=1}^n
  \CH_l(\mathcal{Z}_j) \oplus \bigoplus_{c \in C\backslash U} \CH_l(X_c) \oplus
  \CH_l(D) \longrightarrow \CH_l(X)$$ and it is straightforward to
  conclude.
\end{proof}

The next proposition is a generalisation of Proposition \ref{1} to the
case when the base variety $B$ has dimension greater than $1$.

\begin{proposition} \label{genth} Let $f : X \r B$ be a generically
  smooth, projective morphism defined over $\C$ to a smooth
  quasi-projective variety $B$.  Assume that the singular locus of $f$
  in $B$ is finite and let $U$ be the maximal Zariski-open subset of
  $B$ over which $f$ is smooth.  Assume also that \medskip

  $\bullet$ $\CH_l({X}_b)$ is finitely generated for all closed points
  $b \in B$,

  $\bullet$ $\CH_l({X}_b) \r \mathrm{H}_{2l}(X_b)$ is an isomorphism for all
  closed points $b \in U$,

  $\bullet$ $\CH_{l-i}({X}_{\eta_{D_i}}) = \Q$ for all $i$ such that $0
  < i < d_B$ and all irreducible subvarieties $D_i \subset B$ of
  dimension $i$.

  $\bullet$ $\CH_{l-d_B}({X}_{\eta_{B}})$ is finitely generated. 
  \medskip

  \noindent Then $\CH_l(X)$ has niveau $\leq d_B$.
\end{proposition}

\begin{proof} Let $H \hookrightarrow X$ be a linear section of
  dimension $\geq l+d_B$. The restriction of $f$ to $H$ is surjective.
  Thanks to Lemma \ref{localisation-seq2}, there are finitely many
  closed subschemes $\mathcal{Z}_j$ of $X$ of dimension $l$ such that
  the natural map $\bigoplus_j \CH_l(\mathcal{Z}_j) \oplus
  \bigoplus_{b \in B} \CH_l(X_b) \oplus \CH_l(H) \r \CH_l(X)$ is
  surjective. By Proposition \ref{hilbert}, there exists a closed
  subscheme $\iota : \widetilde{B} \hookrightarrow X$ of dimension
  $d_B+l$ such that the image of the map $\bigoplus_{b \in U}
  \CH_l(X_b) \r \CH_l(X)$ is contained in the image of the map
  $\iota_* : \CH_{l}(\widetilde{B}) \r \CH_l(X)$. Therefore the map
  $$\bigoplus_j \CH_l(\mathcal{Z}_j) \oplus \CH_l(\widetilde{B})
  \oplus \bigoplus_{b \in B\backslash U} \CH_l(X_b) \oplus \CH_l(H)
  \longrightarrow \CH_l(X)$$ is surjective. It is then straightforward
  to conclude.
\end{proof}

\subsection{The main results} The field of complex numbers is a
universal domain and in view of Section \ref{complex} we restate some
of the results above in a more comprehensive way. \medskip

First, we deduce from Proposition \ref{linear-theorem} the following.

\begin{theorem} \label{linear-theorem-complex} Let $f : X \r B$ be a
  complex projective surjective morphism onto a quasi-projective
  variety $B$.  Assume that $\CH_{i}({X}_{b}) = \Q$ for all $i \leq l$
  and all closed point $b \in B$ .  Then
    $\CH_i(X)$ has niveau $\leq d_B$ for all $i \leq l$.
\end{theorem}
\begin{proof}
  By Lemma \ref{L:universal}, $\CH_i(X_{{\eta}_D})=\Q$ for all $i\leq
  l$ and all irreducible subvarieties $D$ of $X$. Proposition
  \ref{linear-theorem} implies that $\CH_i(X)$ has niveau $\leq d_B$.
\end{proof}

The following lemma gives a criterion for the second point in the
assumptions of Propositions \ref{1} and \ref{genth} to be satisfied.

\begin{lemma} \label{Chow-iso}
  Let $X$ be a smooth projective complex variety. Assume that
  $\CH_i(X)$ is finitely generated for all $i \leq l$. Then the cycle
  class map $\CH_l(X) \r \mathrm{H}_{2l}(X)$ is an isomorphism.
\end{lemma}
\begin{proof} The proof follows the same pattern as the proof of
  \cite[Theorem 3.4]{Vial1} once it is noted that if $\CH_i(X)$ is
  finitely generated then the group of algebraically trivial cycles
  $\CH_i(X)_\alg$ is representable in the sense of \cite[Definition
  2.1]{Vial1}. Concretely, we get that the Chow motive of $X$ is
  isomorphic to $\mathds{1} \oplus \mathds{1}(1)^{b_2} \oplus \ldots
  \oplus \mathds{1}(l)^{b_{2l}} \oplus N(l+1)$ where $b_i$ is the
  $i$-th Betti number of $X$ and $N$ is an effective motive. This
  yields that the cycle class map $\CH_i(X) \r \mathrm{H}_{2i}(X)$ is
  an isomorphism for all $i \leq l$.
\end{proof}

As a consequence of Proposition \ref{1}, we obtain

\begin{theorem} \label{1-complex} Let $f : X \r C$ be a complex
  generically smooth projective surjective morphism onto a smooth
  complex curve.  Assume that $\CH_i({X}_c)$ is finitely generated for
  all closed points $c \in C$ and all $i\leq l$, then $\CH_i(X)$ has
  niveau $\leq 1$ for all $i\leq l$.
\end{theorem}
\begin{proof} Let $D$ be an irreducible component of $C$. Lemma
  \ref{L:universal} shows that $\CH_i(X_{\eta_D})$ is finitely
  generated for all $i\leq l$.
  Let $U \subseteq C$ be the smooth locus of $f$. It is an open dense
  subset of $C$. Then, for $c \in U$, the closed fibre $X_c$ is smooth
  and the groups $\CH_i(X_c)$ are finitely generated for all $i \leq
  l$. By Lemma \ref{Chow-iso} the cycle class maps $\CH_i(X_c) \r
  \mathrm{H}_{2i}(X_c)$ are isomorphisms for all $c\in U$ and all $i
  \leq l$.  Proposition \ref{1} implies that $\CH_i(X)$ has niveau
  $\leq 1$.
\end{proof}

And, as a consequence of Proposition \ref{genth}, we obtain

\begin{theorem} \label{genth-complex} Let $f : X \r B$ be a complex
  projective surjective morphism onto a smooth quasi-projective
  variety $B$.  Assume that the singular locus of $f$ in $B$ is
  finite, and assume also that \medskip

 $\bullet$ $\CH_i(X_b)$ is finitely generated for all closed points $b
  \in B$ and all $i \leq l$,

  $\bullet$ $\CH_i({X}_b)=\Q$ for all but finitely many closed points
  $b \in B$ and all $i < l$.  \medskip

  \noindent Then $\CH_i(X)$ has niveau $\leq d_B$ for all $i\leq l$.
\end{theorem}
\begin{proof}
  Let $D$ be an irreducible subvariety of $X$ of positive dimension.
  By Lemma \ref{L:universal}, we have $\CH_i(X_{{\eta}_D})=\Q$ for all
  $i<l$. If we denote by $U$ the smooth locus of $f$, then, by Lemma
  \ref{Chow-iso}, the cycle class maps $\CH_i(X_b) \r
  \mathrm{H}_{2i}(X_b)$ are isomorphisms for all $b\in U$ and all $i
  \leq l$.  Proposition \ref{genth} implies that $\CH_i(X)$ has niveau
  $\leq d_B$ for all $i \leq l$.
\end{proof}

\section{Applications} \label{examples}

\subsection{Varieties with small Chow groups} \label{generalstatements}

In this section, we review the known results about varieties with Chow
groups having small niveau; see Definition \ref{D:niveau}. Varieties
are defined over an algebraically closed field $k$ of characteristic zero and
$\Omega$ denotes a universal domain over $k$. In that case,
Grothendieck's standard conjectures for a smooth projective variety
$X$ over $k$ reduce to the Lefschetz standard conjecture for $X$; see
Kleiman \cite{Kleiman}.

\begin{theorem} \label{T:niveau} Let $X$ be a smooth projective
  variety of dimension $d$.  Assume that the Chow groups
  $\CH_0(X_\Omega), \ldots, \CH_{l}(X_\Omega)$ have niveau $\leq n$.
  \begin{itemize}
  \item If $n=3$ and $l = \lfloor \frac{d-4}{2} \rfloor$, then $X$
  satisfies the Hodge conjecture.
\item If $n=2$ and $l = \lfloor \frac{d-3}{2} \rfloor$, then $X$
satisfies the Lefschetz standard conjecture.
\item If $n=1$ and $l = \lfloor \frac{d-3}{2} \rfloor$, then $X$ has a
  Murre decomposition.
\item If $n=1$ and $l = \lfloor \frac{d-2}{2}  \rfloor$, then $X$ is
Kimura finite-dimensional.
  \end{itemize}
\end{theorem}

\begin{proof}
  The fourth item is proved in \cite{Vial3}. It is also proved there
  that if $X$ is as in the third item, then $X$ has a Chow--K\"unneth
  decomposition. That such a decomposition satisfies Murre's
  conjectures (B), (C) and (D) is proved in \cite[\S
  4.4.2]{Vial2}. The first item is proved in \cite{Laterveer}. We
  couldn't find a reference for the proof of the second item so we
  include a proof here.

  Since it is enough to prove the conclusion of the theorem for
  $X_\Omega$, we may assume that $X$ is defined over $\Omega$. 
  Laterveer used the assumptions on the niveau of the Chow groups to
  show \cite[1.7]{Laterveer} that the diagonal $\Delta_X$ admits a
  decomposition as follows : there exist closed and reduced subschemes
  $V_j, W^j \subset X$ with $\dim V_j \leq j+2$ and $\dim W^j \leq
  n-j$, there exist correspondences $\Gamma_j \in \CH_n(X \times X)$
  for $0 \leq j \leq \lfloor \frac{n-3}{2} \rfloor$ and $\Gamma'\in
  \CH_n(X \times X)$ such that each $\Gamma_j$ is in the image of the
  pushforward map $\CH_n(V_j \times W^j)$, $\Gamma'$ is in the image
  of the pushforward map $\CH_n(X \times W^{\lfloor \frac{n-1}{2}
    \rfloor })$, and
  $$\Delta_X = \Gamma_0 + \ldots + \Gamma_{\lfloor \frac{n-3}{2}
    \rfloor } + \Gamma'.$$ Given $j$ such that $0 \leq j \leq \lfloor
  \frac{n-3}{2} \rfloor$, let $\widetilde{V}_j$ and $\widetilde{W}^j$
  denote desingularisations of ${V}_j$ and ${W}^j$ respectively. The
  action of $\Gamma_j$ on $\mathrm{H}^k(X)$ then factors through
  $\mathrm{H}^k(\widetilde{V}_j)$ and through
  $\mathrm{H}_{2n-k}(\widetilde{W}^j)$. On the one hand, we have
  $\mathrm{H}_{2n-k}(\widetilde{W}^j) =
  \mathrm{H}^{k-2j}(\widetilde{W}^j)$ and hence if $k \leq 2j+1$ then
  the action of $\Gamma_j$ on $\mathrm{H}^k(X)$ factors through the
  $\mathrm{H}^0$ or the $\mathrm{H}^1$ of a smooth projective
  variety. Since the Lefschetz standard conjecture is true in degrees
  $\leq 1$, it follows that the action of $\Gamma_j$ on
  $\mathrm{H}^k(X)$ factors through the $\mathrm{H}_0$ or the
  $\mathrm{H}_1$ of a smooth projective variety. On the other hand, we
  have $\mathrm{H}^{k}(\widetilde{V}_j) =
  \mathrm{H}_{4+2j-k}(\widetilde{V}_j)$ and hence if $k \geq 2j+2$
  then $\Gamma_j$ factors through the $\mathrm{H}_0$, the
  $\mathrm{H}_1$ or the $\mathrm{H}_2$ of a smooth projective
  variety. Concerning the action of $\Gamma'$ on $\mathrm{H}^k(X)$, it
  factors through $\mathrm{H}_{2n-k}(\widetilde{W}^{\lfloor
    \frac{n-1}{2} \rfloor})$ which vanishes for dimension reasons if
  $k < n$ when $n$ is odd and if $k<n-1$ when $n$ is even. When $n$ is
  even and $k=n-1$, the action of $\Gamma'$ on $\mathrm{H}^k(X)$
  factors through the $\mathrm{H}_1$ of a curve.  Indeed this follows
  from a combination of the fact that it factors through
  $\mathrm{H}_{n+1}(\widetilde{W}^{\lfloor \frac{n-1}{2} \rfloor}) =
  \mathrm{H}^1(\widetilde{W}^{ \frac{n-2}{2}})$ and of the validity of
  the Lefschetz standard conjecture in degree $1$.

  By the Lefschetz hyperplane theorem, we get that for $k < n$ the
  cohomology groups $\mathrm{H}^k(X)$ are generated algebraically
  (that is through the action of correspondences) by the
  $\mathrm{H}_0$ of points, the $\mathrm{H}_1$ of curves and the
  $\mathrm{H}_2$ of surfaces. We may then conclude with
  \cite[Proposition 3.19]{Vial4}.
\end{proof}

\begin{remark}
  Note that the second point of Theorem \ref{T:niveau} does not seem
  to follow directly from the method of Arapura \cite{Arapura}, since
  under our assumptions, it is not clear that the middle cohomology
  group of $X$ is motivated by the cohomology of surfaces. Indeed, as
  can be read from the proof of the theorem, when $n$ is even, the
  action of $\Gamma'$ on $\mathrm{H}^n(X)$ factors through
  $\mathrm{H}_{n}(\widetilde{W}^{\lfloor \frac{n-1}{2} \rfloor}) =
  \mathrm{H}^2(\widetilde{W}^{ \frac{n-2}{2}})$.  Although expected by
  the Lefschetz standard conjecture, it is not known if the
  $\mathrm{H}^2$ of a smooth projective variety is motivated by the
  $\mathrm{H}^2$ of a surface.
\end{remark}

\subsection{Varieties fibred by low-degree complete intersections}

As explained by Esnault--Levine--Viehweg in the introduction of
\cite{ELV}, it is expected from general conjectures on algebraic
cycles, that if $Y \subset \P_k^n$ is a complete intersection of
multidegree $d_1 \geq \ldots \geq d_r \geq 2$, then $\CH_l(Y)=\Q$ for
all $l < \lfloor \frac{n-\sum_{i=2}^r d_i}{d_1} \rfloor$, see also
Paranjape \cite{Paranjape} and Schoen \cite{Schoen}. If there is no
proof of the above for the moment, the following theorem however was
proved.

\begin{theorem}[Esnault--Levine--Viehweg \cite{ELV}] \label{ELV} Let
  $Y \subset \P_k^n$ be a complete intersection of multidegree $d_1
  \geq \ldots \geq d_r \geq 2$.
  \begin{itemize}
  \item If either $d_1 \geq 3$ or $r \geq l+1$, assume that
    $\sum_{i=1}^r \left( \begin{array}{c}
        l+d_i  \\
        l+1 \end{array} \right) \leq n$.
\item If $d_1 = \ldots = d_r =2$ and $r \leq l$, assume that
  $\sum_{i=1}^r \left( \begin{array}{c}
      l+d_i  \\
      l+1 \end{array} \right) = r(l+2) \leq n-l+r-1.$
  \end{itemize}
  Then $\CH_{l'}(Y) = \Q$ for all $0 \leq l' \leq l$. \qed
\end{theorem}

Let us consider $f : X \r B$ a dominant morphism between smooth
projective complex varieties whose closed fibres are complete
intersections.  Theorem \ref{linear-theorem-complex}, together with
Theorem \ref{ELV}, shows that the niveau of the first Chow groups of
$X$ have niveau $\leq \dim B$. When $X$ is fibred by very low-degree
complete intersections, we can thus expect $X$ to satisfy the
assumptions of Theorem \ref{T:niveau}.  In the remainder of this
paragraph, we inspect various such cases.

\subsubsection{Varieties fibred by quadric hypersurfaces.}
Let $Q \subset \P^n$ be a quadric hypersurface. Then $\CH_l(Q)=\Q$ for
all $l < \frac{\dim Q}{2}$.

\begin{proposition} \label{quadrics1} Let $f : X \r B$ be a dominant
  morphism between smooth projective complex varieties whose closed
  fibres are quadric hypersurfaces.
  \begin{itemize}
  \item If $\dim B \leq 1$, then $X$ is Kimura finite-dimensional and
    satisfies Murre's conjectures.
    \item  If $\dim B \leq 2$, then $X$ satisfies
    Grothendieck's standard conjectures.
  \item If $\dim B \leq 3$, then $X$ satisfies the Hodge conjecture.
  \end{itemize}
\end{proposition}
\begin{proof} The fibers of $f$ have dimension $\geq \dim X - \dim B$,
  so that $X$ satisfies the assumptions of Theorem
  \ref{linear-theorem-complex} with $l=\lfloor \frac{d_X-d_B - 1}{2}
  \rfloor$. Thus the Chow groups $\CH_0(X)$, $\CH_1(X)$, $\ldots,
  \CH_{\lfloor \frac{d_X-d_B-1}{2} \rfloor}(X)$ have niveau $\leq
  d_B$. We can therefore conclude by Theorem \ref{T:niveau}.
\end{proof}

\subsubsection{Varieties fibred by cubic hypersurfaces.}

Let $X \subset \P^n$ be a cubic hypersurface. Then 
\begin{itemize}
\item $\CH_0(X)=\Q$ for $\dim X \geq 2$.
\item $\CH_1(X)=\Q$ for $\dim X \geq 5$.
\item $\CH_2(X)=\Q$ for $\dim X \geq 8$.
\end{itemize}
Note that Theorem \ref{ELV} only gives $\CH_2(X)=\Q$ for $\dim X \geq
9$. The bound on the dimension of $X$ was improved to $\dim X = 8$ by
Otwinowska \cite{Otwinowska}.

\begin{proposition} \label{cubics} Let $f: X \r B$ be a dominant morphism
  between smooth projective complex varieties whose closed fibres are
  cubic hypersurfaces.
  \begin{itemize}
  \item If $\dim X = 6$ and $\dim B = 1$, then $X$ satisfies
    Grothendieck's standard conjectures and has a Murre decomposition.
  \item If $\dim X = 7$ and if $\dim B \leq 2$, then $X$ satisfies
  the Hodge conjecture.
  \item If $\dim X = 9$ and if $\dim B \leq 1$, then $X$ satisfies
  the Hodge conjecture.
  \end{itemize}
 \end{proposition}
\begin{proof}
  We use Theorem \ref{linear-theorem-complex} as in the proof of
  Proposition \ref{quadrics1}. In the first case, we get that $\CH_0(X)$
  and $\CH_1(X)$ have niveau $\leq 1$. In the second case we get that
  $\CH_0(X)$ and $\CH_1(X)$ have niveau $\leq 2$ and in the third case
  we get that $\CH_0(X)$, $\CH_1(X)$ and $\CH_2(X)$ have niveau $\leq
  1$.  We can then conclude in all three cases by Theorem
  \ref{T:niveau}.
\end{proof}

\subsubsection{Varieties fibred by complete intersections of bidegree
  $(2,2)$.}

Let $X \subset \P^n$ be the complete intersection of two quadrics. By
Theorem \ref{ELV}, $\CH_0(X)=\Q$; and if $\dim X \geq 4$, then
$\CH_1(X)=\Q$.

\begin{proposition} \label{22} Let $f : X \r B$ be a dominant morphism
  between smooth projective complex varieties whose closed fibres are
  complete intersections of bidegree $(2,2)$.
  \begin{itemize}
  \item If $\dim B \leq 1$ and $\dim X \leq 5$, then $X$ is Kimura
    finite-dimensional.
  \item If $\dim B \leq 1$ and $\dim X \leq 6$, then $X$ satisfies
    Murre's conjectures.
  \item If $\dim B \leq 2$ and $\dim X \leq 6$, then $X$ satisfies
    Grothendieck's standard conjectures.
  \item If $\dim B \leq 3$ and $\dim X \leq 7$, then $X$ satisfies the
    Hodge conjecture.
  \end{itemize}
\end{proposition}
\begin{proof} The variety $X$ satisfies the assumptions of Theorem
  \ref{linear-theorem-complex} with $l=1$ for $\dim X - \dim B \geq 4$
  and with $l=0$ in any case. Thus the Chow group $\CH_0(X)$ has
  niveau $\leq d_B$ and $\CH_1(X)$ has niveau $\leq d_B$ for $\dim X -
  \dim B \geq 4$.  We can therefore conclude by Theorem
  \ref{T:niveau}.
\end{proof}

\subsubsection{Varieties fibred by complete intersections of bidegree
  $(2,3)$.}

Let $X \subset \P^n$ be the complete intersection of a quadric and of
a cubic. If $\dim X \geq 6$, then Hirschowitz and Iyer \cite{HI}
showed $\CH_l(X)=\Q$ for $l \leq 1$. (The result of
Esnault--Levine--Viehweg only says that $\CH_l(X)=\Q$ for $l \leq 1$ when
$\dim X \geq 7$).

\begin{proposition} Let $f: X \r C$ be a dominant morphism from a smooth
  projective complex variety $X$ to a smooth projective complex curve
  $C$ whose closed fibres are complete intersections of bidegree
  $(2,3)$ of dimension $6$.  Then $X$ satisfies the Hodge conjecture.
\end{proposition}
\begin{proof}
  By Theorem \ref{linear-theorem-complex}, we see that the Chow groups
  $\CH_0(X)$ and $\CH_1(X)$ have niveau $\leq 1$. We can thus conclude
  by Theorem \ref{T:niveau}.
\end{proof}

\subsection{Varieties fibred by  cellular varieties} \label{cellular}

Let $f : X \r B$ be a complex dominant morphism from a smooth
projective variety $X$ to a smooth projective variety $B$ whose closed
fibres are cellular varieties (not necessarily smooth). In other
words, $X$ is a smooth projective complex variety fibred by cellular
varieties over $B$.  For example, if $\Sigma \subset B$ is the
singular locus of $f$, then $X$ could be such that $X|_{B \backslash
  \Sigma}$ is a rational homogeneous bundle over $B\backslash \Sigma$
(e.g. a Grassmann bundle) and the closed fibres of $f$ over $\Sigma$
(the degenerate fibres) are toric. That kind of situation is
reminiscent of the setting of \cite{GHM2}.

\begin{proposition} \label{thcellular} Let $f : X \r B$ be a dominant
  morphism between smooth projective complex varieties whose closed
  fibres are cellular varieties.
  \begin{itemize}
  \item Assume $B$ is a curve, then $X$ is Kimura finite-dimensional
    and $X$ satisfies Murre's conjectures.
  \item Assume $\dim B \leq 2 $ and $\dim X \leq 6$. If $f$ is
    connected and smooth away from finitely many points in $B$, then
    $X$ satisfies Grothendieck's standard conjectures.
  \item Assume $\dim B \leq 3$ and $\dim X \leq 7$.  If $f$ is
    connected and smooth away from finitely many points in $B$, then
    $X$ satisfies the Hodge conjecture.
 \end{itemize}
 \end{proposition}

 \begin{proof} The Chow groups of cellular varieties are finitely
   generated. The first statement thus follows from Theorems
   \ref{1-complex} and \ref{T:niveau}.
   Let us now focus on the cases when $\dim B$ is either $2$ or $3$.  It
  is a consequence of Mumford's theorem \cite{Mumford} that a
  connected smooth projective complex variety with finitely generated
  Chow group of zero-cycles actually has Chow group of zero-cycles
  generated by a point.  Thus the second and third statements follow
  from Theorem \ref{genth-complex} with $l=1$, and from Theorem
  \ref{T:niveau}.
\end{proof}

\section*{Acknowledgements} Thanks to Reza Akhtar and Roy Joshua for
asking a question related to the Chow--K\"unneth decomposition of a
smooth blow-up, to Sergey Gorchinskiy for suggesting parts of
Proposition \ref{general-generic}, to Mingmin Shen for a discussion
related to the proof of Lemma \ref{locally-closed}, and to Burt Totaro
for suggesting the statement of Lemma \ref{lemma}. Finally, many
thanks to the referee for helpful comments and suggestions.


\vspace{5pt}

\begin{thebibliography}{10}

\bibitem{Arapura}
D.~Arapura.
\newblock Motivation for {H}odge cycles.
\newblock {\em Adv. Math.}, 207(2):762--781, 2006.

\bibitem{B-conic}
A.~Beauville.
\newblock Vari\'et\'es de {P}rym et jacobiennes interm\'ediaires.
\newblock {\em Ann. Sci. \'Ecole Norm. Sup. (4)}, 10(3):309--391, 1977.

\bibitem{BS}
S.~Bloch and V.~Srinivas.
\newblock Remarks on correspondences and algebraic cycles.
\newblock {\em Amer. J. Math.}, 105(5):1235--1253, 1983.

\bibitem{Bloch-lectures}
S.~Bloch.
\newblock {\em Lectures on algebraic cycles}, volume~16 of {\em New
  Mathematical Monographs}.
\newblock Cambridge University Press, Cambridge, second edition, 2010.

\bibitem{Bouali}
J.~Bouali.
\newblock { Motives of quadric bundles}.
\newblock Preprint.

\bibitem{dAMS}
P.L.~del Angel and S.~M{\"u}ller-Stach.
\newblock Motives of uniruled {$3$}-folds.
\newblock {\em Compositio Math.}, 112(1):1--16, 1998.

\bibitem{ELV}
H.~Esnault, M.~Levine, and E.~Viehweg.
\newblock Chow groups of projective varieties of very small degree.
\newblock {\em Duke Math. J.}, 87(1):29--58, 1997.

\bibitem{Fulton} W.~Fulton.  \newblock {\em Intersection theory},
  volume~2 of {\em Ergebnisse der Mathematik und ihrer Grenzgebiete.
    3. Folge. A Series of Modern Surveys in Mathematics}. 
\newblock Springer-Verlag, Berlin, second edition, 1998.

\bibitem{GHM2}
B.B. Gordon, M.~Hanamura, and J.P.~Murre.
\newblock Relative {C}how-{K}\"unneth projectors for modular varieties.
\newblock {\em J. Reine Angew. Math.}, 558:1--14, 2003.

\bibitem{HI}
A.~Hirschowitz and J.~Iyer.
\newblock Hilbert schemes of fat {$r$}-planes and the triviality of {C}how
  groups of complete intersections.
\newblock In {\em Vector bundles and complex geometry}, volume 522 of {\em
  Contemp. Math.}, pages 53--70. Amer. Math. Soc., Providence, RI, 2010.

\bibitem{Iyer2}
J.~Iyer.
\newblock Absolute {C}how-{K}\"unneth decomposition for rational homogeneous
  bundles and for log homogeneous varieties.
\newblock {\em Michigan Math. J.}, 60(1):79--91, 2011.

\bibitem{Jannsen} U.~Jannsen.  
\newblock Motivic sheaves and filtrations on {C}how groups.
\newblock In {\em Motives (Seattle,
    WA, 1991)}, volume~55 of {\em Proc. Sympos.  Pure Math.}, pages
  245--302. Amer. Math. Soc., Providence, RI, 1994.

\bibitem{Kimura}
S.-I.~Kimura.
\newblock Chow groups are finite dimensional, in some sense.
\newblock {\em Math. Ann.}, 331(1):173--201, 2005.

\bibitem{Kleiman}
S.L.~Kleiman.
\newblock The standard conjectures.
\newblock In {\em Motives ({S}eattle, {WA}, 1991)}, volume~55 of {\em Proc.
  Sympos. Pure Math.}, pages 3--20. Amer. Math. Soc., Providence, RI, 1994.

\bibitem{Kollar}  J. Koll\'ar.
\newblock{\em Rational curves on algebraic varieties}, volume~32 of {\em
Ergebnisse der Mathematik und ihrer Grenzgebiete. 3. Folge. A Series
of Modern Surveys in Mathematics}. 
\newblock Springer-Verlag, Berlin, 1996.


\bibitem{Laterveer}
R.~Laterveer.
\newblock Algebraic varieties with small {C}how groups.
\newblock {\em J. Math. Kyoto Univ.}, 38(4):673--694, 1998.

\bibitem{Manin}
Yu.I.~Manin.
\newblock Correspondences, motifs and monoidal transformations.
\newblock {\em Mat. USSR-Sb.}, 6:439--470, 1968.

\bibitem{Mumford}
D.~Mumford.
\newblock Rational equivalence of {$0$}-cycles on surfaces.
\newblock {\em J. Math. Kyoto Univ.}, 9:195--204, 1968.

\bibitem{Murre1}
J.P.~Murre.
\newblock On a conjectural filtration on the {C}how groups of an algebraic
  variety. {I}. {T}he general conjectures and some examples.
\newblock {\em Indag. Math. (N.S.)}, 4(2):177--188, 1993.

\bibitem{Murre2}
J.P.~Murre.
\newblock On a conjectural filtration on the {C}how groups of an
algebraic variety. {II}. {V}erification of the conjectures for
threefolds which are the product on a surface and a curve.
\newblock {\em Indag. Math. (N.S.)}, 4(2):189--201, 1993.

\bibitem{Nagel-Saito}
J.~Nagel and M.~Saito.
\newblock Relative {C}how-{K}\"unneth decompositions for conic bundles and
  {P}rym varieties.
\newblock {\em Int. Math. Res. Not. IMRN}, (16):2978--3001, 2009.

\bibitem{Otwinowska}
A.~Otwinowska.
\newblock Remarques sur les groupes de {C}how des hypersurfaces de petit
  degr\'e.
\newblock {\em C. R. Acad. Sci. Paris S\'er. I Math.}, 329(1):51--56, 1999.

\bibitem{Paranjape}
K.H.~Paranjape.
\newblock Cohomological and cycle-theoretic connectivity.
\newblock {\em Ann. of Math. (2)}, 139(3):641--660, 1994.

\bibitem{Schoen}
C.~Schoen.
\newblock On {H}odge structures and nonrepresentability of {C}how groups.
\newblock {\em Compositio Math.}, 88(3):285--316, 1993.

\bibitem{VialCK}
C.~Vial.
\newblock {Chow-Kuenneth decomposition for $3$- and $4$-folds fibred by
  varieties with trivial Chow group of zero-cycles}.
\newblock {J. Algebraic Geom., to appear}.

\bibitem{Vial2}
C.~Vial.
\newblock {Niveau and coniveau filtrations on cohomology groups and Chow
  groups}.
\newblock { Proc. Lond. Math. Soc. (3)}, 106(2):410--444, 2013.

\bibitem{Vial3}
C.~Vial.
\newblock {Projectors on the intermediate algebraic Jacobians}.
\newblock {Preprint}.

\bibitem{Vial4}
C.~Vial.
\newblock {Remarks on motives of abelian type}.
\newblock {Preprint, 2012}.

\bibitem{Vial1}
C.~Vial.
\newblock Pure motives with representable {C}how groups.
\newblock {\em C. R. Math. Acad. Sci. Paris}, 348(21-22):1191--1195, 2010.

\end{thebibliography}

\end{document}